\documentclass[a4paper,11pt]{article}
\usepackage{amsmath, amsfonts, amscd, amssymb, amsthm, enumerate, xypic}

\def\defterm{\textbf}

\newcommand{\iadj}{\;\underset{\text{i}}{\rightarrow}\;}
\newcommand{\uadj}{\;\underset{\text{u}}{\rightarrow}\;}

\newcommand{\Mat}{\operatorname{M}}
\newcommand{\charac}{\operatorname{char}}

\newcommand{\id}{\operatorname{id}}
\newcommand{\GL}{\operatorname{GL}}
\newcommand{\GP}{\operatorname{GP}}
\newcommand{\SP}{\operatorname{SP}}
\newcommand{\SL}{\operatorname{SL}}
\newcommand{\Ker}{\operatorname{Ker}}
\newcommand{\End}{\operatorname{End}}

\newcommand{\Vect}{\operatorname{span}}
\newcommand{\im}{\operatorname{Im}}

\newcommand{\tr}{\operatorname{tr}}

\renewcommand{\setminus}{\smallsetminus}

\def\F{\mathbb{F}}

\def\C{\mathbb{C}}

\def\N{\mathbb{N}}
\def\Z{\mathbb{Z}}

\def\calA{\mathcal{A}}

\def\calC{\mathcal{C}}

\def\calE{\mathcal{E}}
\def\calF{\mathcal{F}}

\def\calH{\mathcal{H}}

\def\calU{\mathcal{U}}
\def\calV{\mathcal{V}}

\def\lcro{\mathopen{[\![}}
\def\rcro{\mathclose{]\!]}}

\theoremstyle{definition}
\newtheorem{Def}{Definition}[section]
\newtheorem{Not}[Def]{Notation}

\theoremstyle{plain}
\newtheorem{theo}{Theorem}[section]
\newtheorem{prop}[theo]{Proposition}
\newtheorem{cor}[theo]{Corollary}
\newtheorem{lemma}[theo]{Lemma}

\theoremstyle{plain}

\theoremstyle{remark}
\newtheorem{Rems}{Remarks}

\title{Products of involutions in the stable general linear group}
\author{Cl\'ement de Seguins Pazzis\footnote{Universit\'e de Versailles Saint-Quentin-en-Yvelines, Laboratoire de Math\'ematiques
de Versailles, 45 avenue des Etats-Unis, 78035 Versailles cedex, France}
\footnote{e-mail address: dsp.prof@gmail.com}}

\begin{document}

\thispagestyle{plain}

\maketitle


\begin{abstract}
In the stable general linear group over an arbitrary field, we prove that every element with determinant $\pm 1$
is the product of three involutions, and of no less in general. We also obtain several results of the same flavor, with applications to decompositions of
automorphisms of an infinite-dimensional vector space that are scalar multiples of finite-rank perturbations of the identity.
\end{abstract}




\section{Introduction}

\subsection{The problem}

Let $\F$ be a field, whose group of units we denote by $\F^*$. Denote by $\Mat_n(\F)$ the algebra of all $n$ by $n$
square matrices with entries in $\F$, by $\GL_n(\F)$ its group of invertible elements, and by
$\SL_n(\F)$ its subgroup of all matrices with determinant~$1$.
The zero matrix of $\Mat_n(\F)$ is denoted by $0_n$, the identity matrix by $I_n$.
A matrix of $\Mat_n(\F)$ will be called \textbf{scalar} when it is a scalar multiple of $I_n$.

An element $x$ of a group $G$, with unity $1_G$, is called an \textbf{involution} whenever $x^2=1_G$.
An element $x$ of a unital ring $R$ is called \textbf{unipotent of index $2$} when $(x-1_R)^2=0_R$
(i.e.\ it is invertible, with inverse $2.1_R-x$).
In particular, a matrix $A \in \GL_n(\F)$ is an involution if and only if $A^2=I_n$, and it is
unipotent of index $2$ if and only if $(A-I_n)^2=0_n$ (in which case we say that $A$ is a $U_2$\textbf{-matrix}).
We note that the $U_2$-matrices are the involutions if $\F$ has characteristic $2$.
Every involutary matrix has determinant $\pm 1$, while every $U_2$-matrix has determinant $1$. Note also that
$I_n$ is both an involution and a $U_2$-matrix.

Our starting point is the classical problem of decomposing a square matrix into a product of involutions (with unprescribed number of factors).
Obviously, a matrix that is a product of involutions must be invertible, and more precisely its determinant must equal $\pm 1$.
The converse is easily proved by noting that any transvection matrix is the product of two involutions:
for $2$ by $2$ matrices, we note that, for all $\lambda \in \F$,
$$\begin{bmatrix}
1 & \lambda \\
0 & 1
\end{bmatrix}=\begin{bmatrix}
1 & 0 \\
0 & -1
\end{bmatrix}\begin{bmatrix}
1 & \lambda \\
0 & -1
\end{bmatrix}.$$

The next step in this problem is the so-called \emph{length problem}: given a matrix $A \in \GL_n(\F)$ with determinant $\pm 1$,
what is the minimal number of factors $\ell(A)$ (called the length of $A$) required to write $A$ as the product of $\ell(A)$ involutions?
Surprisingly, $\ell(A)$ is very small! More precisely, $\ell(A) \leq 4$; in other words, every matrix with determinant $\pm 1$
is the product of at most four involutions (see \cite{Gustafsonetal}, and \cite{Sourour} for a shorter proof over fields with large cardinality).
Yet, in general there are matrices with determinant $\pm 1$
that fail to be the product of three involutions (e.g.\ any matrix of the form $\alpha I_n$ in which $\alpha \in \F^*$
satisfies $\alpha^n=\pm 1$ and $\alpha^4 \neq 1$; see \cite{HalmosKakutani}).

The matrices that are the product of two involutions are known:
the celebrated theorem of Wonenburger \cite{Wonenburger} (for the field of complex numbers), Djokovic \cite{Djokovic}
(for the general case), and Hoffmann and Paige \cite{HoffmanPaige} (an independent discovery) states
that they are exactly the invertible matrices that are similar to their
inverse: this result is rephrased as point (i) of Theorem \ref{theo2} in the present article.
Note that, in any group, an element is the product of two involutions only if it is conjugated to its inverse.

The remaining open problem is the determination of the matrices $A$ with length $3$. Of course, the length of $A$ is invariant under
conjugation, and hence it is encoded in the invariant factors of $A$ (i.e.\ its rational canonical form).
Unfortunately, several studies in low dimension have shown that for length $3$ no neat necessary and sufficient condition in terms of
invariant factors appears possible (a famous quote by Paul Halmos even states that ``the best known characterization of products of
three involutions is being the product of three involutions"). Several interesting non-trivial necessary conditions have been found however:
for example, if $A$ is of length $3$ then it has no eigenvalue $\lambda$ with geometric multiplicity at least
$\frac{3n}{4}$ and such that $\lambda^4 \neq 1$
(see \cite{Ballantine}). This result has been improved by Liu (see Theorem 3.1 of \cite{Liu}).
Moreover, several nice sufficient conditions are also known: for example if $A$ has determinant $\pm 1$ and a sole invariant factor
then it is the product of three involutions (see \cite{Ballantine}, and also Proposition \ref{cyclicdecomp} here);
if $\F$ is the field of complex numbers, $A$ has determinant $\pm 1$
and all its eigenvalues have geometric multiplicity at most $\frac{n}{2}$, then $A$ is the product of three involutions
\cite{Liu}.
Here, we will prove a variation of that result for arbitrary fields (see Theorem \ref{unstabletheo3}).
Finally, characterizations are known for very small values of $n$.
Yet, we agree with Halmos that a full solution to the length problem should be viewed as an essentially hopeless endeavour.

In the present article,  we will not tackle the length problem \emph{per se} but the \emph{stable} length problem, which is
motivated by the length problem in the general linear group of an infinite-dimensional vector space (see Section \ref{infinitedimSection}).
Given $A \in \GL_n(\F)$ and $p \in \N$, we consider the ``augmented matrix"
$$A \oplus I_p:=\begin{bmatrix}
A & 0 \\
0 & I_p
\end{bmatrix} \in \GL_{n+p}(\F).$$
Interestingly, this new matrix has the same determinant as $A$, and hence it is a product of involutions if and only if so is $A$.
Obviously, if $A$ is of length $k$ then $A \oplus I_p$
is of length at most $k$: indeed if we split $A=S_1 \cdots S_k$ for involutions $S_1,\dots,S_k$, then
$A \oplus I_p =(S_1 \oplus I_p) \cdots (S_k \oplus I_p)$ is obviously the product of $k$ involutions.
Moreover, judging from Djokovic's theorem,
$A \oplus I_p$ is of length $2$ if and only if $A$ is of length $2$
(classically, the primary canonical form yields a cancellation rule for the similarity of matrices with respect to the direct sum).
Strikingly, there are cases when $\ell(A \oplus I_p)=3$ whereas $\ell(A)=4$!
For example, it is known that given a positive integer $p>0$ and a scalar $\alpha$ in $\F$ with $\alpha^p=\pm 1$ and $\alpha^4\neq 1$,
the matrix $\alpha I_p$ is of length $4$ (see \cite{HalmosKakutani}), yet $\alpha I_p \oplus I_p$ is of length $3$
 (see Lemma \ref{diagonal3invollemmafinal}).

Here, we shall prove that for every matrix $A \in \GL_n(\F)$ having determinant $\pm 1$,
the augmented matrix $A \oplus I_n$ is the product of three involutions.
In \cite{dSP3square}, a similar result was proved for the decomposition of a trace-zero matrix into the sum of three square-zero matrices.
The striking point here is that, in the known sufficient conditions for a matrix $A \in \GL_n(\C)$ with determinant
$\pm 1$ to be the product of three involutions, it is required that there be no eigenvalue with geometric multiplicity too large.
In contrast, here it is precisely the fact that $1$ is an eigenvalue with large geometric multiplicity that will make $A\oplus I_n$
a product of three involutions if $\det A=\pm 1$!

\vskip 3mm
The stable length problem has a nice reformulation as a statement on the stable general linear group.
Recall that this group can be defined as follows.
For $A \in \GL_n(\F)$ and $B \in \GL_p(\F)$, we say that $A$ and $B$ are stably equal whenever $A \oplus I_p=B \oplus I_n$.
This defines an equivalence relation on the union $\underset{n \in \N}{\bigcup} \GL_n(\F)$, whose quotient
set we denote by $\GL_\infty(\F)$. Noting that the class of the product $(A \oplus I_p) \times (B \oplus I_n)$
depends only on the respective classes of the matrices $A \in \GL_n(\F)$ and $B \in \GL_p(\F)$, we naturally endow
$\GL_\infty(\F)$ with a group structure.
Noting that $\det(A \oplus I_1)=\det(A)$ for all $A \in \GL_n(\F)$, we see that all the matrices in an equivalence class
share the same
determinant. This yields a group homomorphism from $\GL_\infty(\F)$ to $\F^*$, called the determinant.

Let now $A \in \GL_n(\F)$ have determinant $\pm 1$.
By the above, the sequence of lengths $\bigl(\ell(A\oplus I_k)\bigr)_{k \in \N}$ is non-increasing,
and one sees that its ultimate value is the length of the class of $A$ in $\GL_\infty(\F)$, i.e.\
the minimal number of factors required to write this class as a product of involutions.
Moreover, this length equals $2$ if and only if the length of $A$ equals $2$, which is equivalent to the class of $A$
being conjugated to its inverse in $\GL_\infty(\F)$.
Hence, as a consequence of Theorems \ref{theo2} and \ref{theo3invol} that follow, the length problem will be completely solved in the stable group $\GL_\infty(\F)$:

\begin{theo}
\begin{enumerate}[(a)]
\item An element of $\GL_\infty(\F)$ is a product of involutions if and only if its determinant equals $\pm 1$.
\item An element of $\GL_\infty(\F)$ is the product of two involutions if and only if it is conjugated to its inverse.
\item Every element of $\GL_\infty(\F)$ with determinant $\pm 1$ is the product of three involutions.
\end{enumerate}
\end{theo}

\vskip 3mm
Actually, we will not restrict our study to decompositions into products of involutions, because
the techniques we develop here allow us to consider more general decompositions that involve
involutions and $U_2$-matrices. Here is our more general problem: let
$\calA_1 ,\dots,\calA_k$ be subsets of $\GL_n(\F)$, each of which equal to the set of all involutions or to the set of all $U_2$-matrices, and set
$$\calA_1\cdots \calA_k:=\Biggl\{\prod_{i=1}^k S_i \mid S_1 \in \calA_1,\dots,S_k \in \calA_k\Biggr\}.$$

Given a matrix $A \in \GL_n(\F)$, can we give a nice necessary and sufficient condition for $A$ to belong to
$\calA_1 \cdots \calA_k$? A full solution to this is known when $k=2$, and we will also obtain one for $k \geq 4$.
A complete solution in the case $k=3$ is of course out of reach as it would imply a characterization of products of three involutions.
In the case $k=3$, we will however give a complete solution to the stable version of this problem (see Theorems \ref{theo3invol} to \ref{theo3mixed2}).

In general, we note that, since each set $\calA_i$ is stable under conjugation, so is $\calA_1\cdots \calA_k$.

Moreover, it is crucial to observe that the order of factors is not important. To see this, consider two subsets
$\calU$ and $\calV$ of $\GL_n(\F)$ that are both stable under conjugation and transposition. Then, we claim that
$\calU\calV=\calV \calU$. First, $\calV \calU$ is stable under conjugation, obviously.
Then, given $(u,v)\in \calU \times \calV$, we write $(uv)^T=v^T u^T$ to find that $(uv)^T$ belongs to $\calV \calU$
and we conclude that so does $uv$ because every square matrix with entries in a field is similar to its transpose.
It follows that $\calA_1\cdots \calA_k=\calA_{\sigma(1)}\cdots \calA_{\sigma(k)}$ for every permutation $\sigma$ of $\lcro 1,k\rcro$.

In particular, given $k \in \{0,1,2,3\}$, if a matrix of $\GL_n(\F)$
is the product of $k$ involutions and $(3-k)$ unipotent matrices of index $2$ in some prescribed order,
then it is the product of $k$ involutions and $(3-k)$ unipotent matrices of index $2$ in \emph{any} possible order!

Hence, for the length $3$ problem, we only have four cases to consider, and for the length $4$ problem only five cases
need consideration.

\subsection{Main results}

It is time to state our main results. Here, we write $A \simeq B$ to state that two square matrices $A$ and $B$ are similar.
We start by recalling the characterization of products of two involutions, and the one of products of two $U_2$-matrices.
We will make systematic use of them.
See \cite{Djokovic} for statement (i), and \cite{Bothaunipotent} for statement (ii)
(see also the recent \cite{dSPprod2} for a more general characterization of products of two invertible matrices
with prescribed annihilated polynomials with degree $2$).

\begin{theo}\label{theo2}
Let $M \in \GL_n(\F)$.
\begin{enumerate}[(i)]
\item The matrix $M$ is the product of two involutions if and only if
$M \simeq M^{-1}$.
\item The matrix $M$ is the product of two $U_2$-matrices if and only if
$M \simeq M^{-1}$ and, if $\charac(\F) \neq 2$, all the Jordan cells of $M$ with respect to the eigenvalue $-1$
are even-sized.
\end{enumerate}
\end{theo}

Note in particular that a matrix that is the product of two $U_2$-matrices is also the product of two involutions!

The matrices that are the product of an involution and a $U_2$-matrix are also known: see
\cite{Wang1} for the field of complex numbers, and \cite{dSPprod2} for the general case.
We will only use the following two sufficient conditions:

\begin{theo}\label{theo2mixed}
Let $M \in \GL_n(\F)$. Assume that $M \simeq -M^{-1}$ and that,
for any $\alpha \in \F \setminus \{1\}$ such that $\alpha^2=-1$, the Jordan cells
of $M$ associated to the eigenvalue $\alpha$ are all even-sized. Then, $M$
is the product of a $U_2$-matrix and an involution.
\end{theo}

\begin{theo}\label{theo2mixedbis}
Let $k,l$ be non-negative integers such that $|k-l| \leq 2$, and let $M$ be
the direct sum of a Jordan cell with size $k$ for the eigenvalue $1$ and of a Jordan cell with size $l$
for the eigenvalue $-1$. Then, $M$ is the product
of a $U_2$-matrix and an involution.
\end{theo}

Now, we turn to the new results. First, our result on the length $4$ problem in the general linear group (not the stable one!):

\begin{theo}\label{theo4}
Let $\calA_1,\dots,\calA_4$ be subsets of $\GL_n(\F)$, in which each $\calA_i$ equals the
set of all involutions or the one of all $U_2$-matrices.

If at least one $\calA_i$ equals the set of all involutions, then
$$\calA_1\calA_3\calA_3\calA_4=\bigl\{M \in \GL_n(\F) : \; \det M=\pm 1\}.$$
Otherwise,
$$\calA_1\calA_3\calA_3\calA_4=\SL_n(\F).$$
\end{theo}

Here, the case when all the $\calA_i$'s equal the set of all involutions of $\GL_n(\F)$ was already known, as stated in the introduction
(\cite{Gustafsonetal}), and the case when all the $\calA_i$'s equal the set of all $U_2$-matrices of $\GL_n(\F)$ was known
over the field of complex numbers (see \cite{WangWu}).

Next, we have a new sufficient condition for the decomposability into the product of three matrices, either
unipotent of index $2$ or involutary:

\begin{theo}\label{unstabletheo3}
Let $M \in \GL_n(\F)$ be such that $\det M=\pm 1$.
Assume that $M$ has at most one Jordan cell of size $1$ for each one of its eigenvalues in $\F$,
and that the characteristic polynomial of $M$ is not a power of some irreducible polynomial.

Then, for all $k \in \{0,1,2\}$, the matrix $M$ is the product of $k$ unipotent matrices of index $2$
and $3-k$ involutions.
Moreover, if $\det M=1$ then $M$ is the product of three $U_2$-matrices.
\end{theo}

Note that our assumptions imply that every eigenvalue of $M$ in $\F$ has geometric multiplicity at most $\frac{n}{2}\cdot$
Hence, in the case when $k=0$ and $\F$ is the field of complex numbers, our result is weaker than the
result of Liu recalled in the introduction (theorem 2.5 of \cite{Liu}).

Our most demanding results deal with the stable length $3$ problem.
First, the case of three involutions:

\begin{theo}\label{theo3invol}
Let $A \in \GL_n(\F)$ have determinant $\pm 1$. Then, $A \oplus I_n$ is the product of three involutions.
\end{theo}

Then, the case of three $U_2$-matrices, which turns out to be easier to deal with:

\begin{theo}\label{theo3U2}
Let $A \in \SL_n(\F)$. Then, $A \oplus I_n$ is the product of three $U_2$-matrices.
\end{theo}

Finally, the results on ``mixed" products, the latter of which is the most difficult of all:

\begin{theo}\label{theo3mixed1}
Let $A \in \GL_n(\F)$ have determinant $\pm 1$.
Then, $A \oplus I_n$ is the product of two involutions and one $U_2$-matrix.
\end{theo}

\begin{theo}\label{theo3mixed2}
Let $A \in \GL_n(\F)$ have determinant $\pm 1$.
Then, $A \oplus I_n$ is the product of one involution and two $U_2$-matrices.
\end{theo}

Using the same techniques, we will also prove three additional results of the same flavor in which we augment the matrix $A$
not by an identity matrix, but by a scalar multiple of an identity matrix.

The motivation for tackling such results is related to the characterization of the scalar matrices that are of length $3$.
It can indeed be proved that, given a scalar $\alpha$ and a positive integer $n$:
\begin{itemize}
\item The matrix $\alpha I_n$ is the product of three involutions if and only if $\alpha=\pm 1$, or
$\alpha^4=1$ and $n$ is even. The same holds for the decomposition into the product of one involution and two $U_2$-matrices.
\item The matrix $\alpha I_n$ is the product of three $U_2$-matrices if and only if $\alpha=1$, or
$\alpha=-1$ and $n$ is even.
\item The matrix $\alpha I_n$ is the product of two involutions and one $U_2$-matrix if and only if $\alpha=\pm 1$.
\end{itemize}

\begin{theo}\label{skew3involutions}
Assume that $\F$ has characteristic not $2$, and let $i$ be an element of $\F$ such that $i^2=-1$.
Let $A \in \GL_n(\F)$, and let $r \geq n$ be an integer such that $i^r \det A=\pm 1$.
Then, $A \oplus (i I_r)$ is the product of three involutions.
\end{theo}

\begin{theo}\label{skew3unipotent}
Assume that $\F$ has characteristic not $2$.
Let $A \in \GL_n(\F)$, and let $k \geq n$ be an integer such that $(-1)^k \det A=1$.
Then, $A \oplus (- I_k)$ is the product of three $U_2$-matrices.
\end{theo}

\begin{theo}\label{skew1involution2unipotents}
Assume that $\F$ has characteristic not $2$, and let $i$ be an element of $\F$ such that $i^2=-1$.
Let $A \in \GL_n(\F)$, and let $r \geq n$ be an integer such that $i^r \det A=\pm 1$.
Then, $A \oplus (i I_r)$ is the product of one involution and two $U_2$-matrices.
\end{theo}

\subsection{Application to the general linear group of an infinite-dimensional vector space}\label{infinitedimSection}

Our motivation for tackling the stable length problem comes from the length problem in infinite-dimensional vector spaces.
Let $V$ be an infinite-dimensional vector space over $\F$. Denote by $\End(V)$ the algebra of all endomorphisms of $V$, and by
$\GL(V)$ its group of invertible elements (i.e.\ the automorphisms of $V$).
It can be shown that every element of $\GL(V)$ is a product of involutions, and, better still,
every element of $\GL(V)$ is the product of four involutions (this will be proved in a subsequent article).
Over fields with more than $3$ elements, there are automorphisms that are not the product of three involutions however,
which motivates us to characterize the automorphisms that are the product of three involutions.

In considering this problem, it turns out that a special kind of automorphism needs to be singled out:
the ones that equal $\alpha \id_V+u$ for some nonzero scalar $\alpha$ and some \emph{finite-rank} endomorphism $u$.
Denote by $\End_f(V)$ the two-sided ideal of $\End(V)$ consisting of the finite-rank endomorphisms of $V$.
Then, $\F\id_V \oplus \End_f(V)$ is a subalgebra of $\End(V)$, denoted by $\calA(V)$, and every element of it that is invertible in $\End(V)$
has its inverse in $\calA(V)$. To every $f \in \calA(V)$, we assign the unique $\lambda(f) \in \F$
such that $f-\lambda(f) \id_V$ has finite rank, thereby defining a morphism of $\F$-algebras from $\calA(V)$ to $\F$.
We denote by $\GP_f(V)$ the group of all invertible elements of the algebra $\calA(V)$, and by
$\SP_f(V)$ the subgroup of all elements of $\GP_f(V)$ of the form $\id_V+u$ for some $u \in \End_f(V)$
(i.e.\ the kernel of $f \in \GP_f(V) \mapsto \lambda(f) \in \F^*$).
Hence, $\GP_f(V)$ is isomorphic to the direct product of $\SP_f(V)$ with $\F^*$.

For every $u \in \SP_f(V)$ and every finite-dimensional linear subspace $W$ of $V$ that includes $\im(u-\id_V)$,
the determinant of the induced endomorphism $u_{|W}$ depends only on $u$ (not on the choice of $W$).
By assigning this quantity to $u$, one obtains a group homomorphism from $\SP_f(V)$ to $\F^*$, called the determinant.

Here, we shall derive the following results from the theorems stated in the preceding section:

\begin{prop}
Let $u \in \SP_f(V)$ have determinant $\pm 1$.
Then, in the algebra $\calA(V)$, $u$ is the product of three involutions, and also of
one unipotent element of index $2$ and two involutions, and also
of one involution and two unipotent elements of index $2$ (in any prescribed order).

Moreover, if $u$ has determinant $1$ then it is the product of three unipotent elements of index $2$.
\end{prop}

\begin{proof}
We prove the first claimed result. The proof is similar for the other three, and consequently left to the reader.

We choose a finite-dimensional linear subspace $W$ of $V$ such that $\im(u-\id_V) \subset W$ and $W +\Ker(u-\id_V)=V$.
Then, we choose a linear subspace $H$ of $\Ker(u-\id_V)$ such that $W \oplus H=V$. Set $n:=\dim W$.
Then, $H$ is infinite-dimensional, and hence we can re-split $H=H_1 \oplus H_2$ where $\dim H_1=n$.
Choose a matrix $A$ that represents the automorphism of $W$ induced by $u$. Then, $\det A=\det u=\pm 1$.
Since $u$ is the identity on $H_1$, the automorphism $v$ of $W \oplus H_1$ induced by $u$ is represented by
$A \oplus I_n$ in some basis. Hence, by Theorem \ref{theo3invol}, $v=abc$ for some involutions $a,b,c$ in $\GL(W \oplus H_1)$.
Now, extend $a,b,c$ to automorphisms $\tilde{a},\tilde{b},\tilde{c}$ of $V$ that are the identity on $H_2$.
Obviously, $\tilde{a},\tilde{b},\tilde{c}$ are involutions that belong to $\SP_f(V)$, and $u=\widetilde{a}\widetilde{b}\widetilde{c}$.
\end{proof}

Noting that the opposite of an involution is an involution, we deduce the following corollary:

\begin{cor}
Let $u \in \SP_f(V)$ have determinant $\pm 1$, and let $\varepsilon \in \{-1,1\}$.
Then, in the algebra $\End(V)$, $\varepsilon u$ is the product of three involutions, and also of
one unipotent element of index $2$ and two involutions, and also
of one involution and two unipotent elements of index $2$ (in any prescribed order).
\end{cor}

Here are the corresponding results for special extensions:

\begin{prop}\label{infinitedimi}
Let $i \in \F^*$ be of order $4$.
Let $u \in \SP_f(V)$ have its determinant in $\{\pm 1,\pm i\}$.
Then, in the algebra $\End(V)$, the automorphism $iu$ is the product of three involutions,
and also the product of one involution and two unipotent elements of index $2$ (in any prescribed order).
\end{prop}

\begin{prop}\label{infinitedim-1}
Let $u \in \SP_f(V)$ have determinant $\pm 1$.
Then, in the algebra $\End(V)$, $-u$ is the product of three unipotent elements of index $2$.
\end{prop}

We only prove Proposition \ref{infinitedimi}, since the proof of Proposition \ref{infinitedim-1} is essentially similar.

\begin{proof}
We choose a finite-dimensional linear subspace $W$ of $V$ such that $\im(u-\id_V) \subset W$ and $W +\Ker(u-\id_V)=V$.
Then, we choose a linear subspace $H$ of $\Ker(u-\id_V)$ such that $W \oplus H=V$. Set $n:=\dim W$.

Choose a matrix $A$ that represents the automorphism of $W$ induced by $u$. Then $\det (iA)=i^n \det u \in \{\pm 1,\pm i\}$.
By Theorem \ref{skew3involutions}, there is an integer $k \geq 0$ such that $i A \oplus i I_k$ is the product of three involutions in $\GL_{n+k}(\F)$.
Then, we resplit $H=H_1 \oplus H_2$ where $\dim H_1=k$.

Since $u$ is the identity on $H_1$, the automorphism $v$ of $W \oplus H_1$ induced by $i u$ is represented by
$i A \oplus i I_k$ in some basis. Hence, $v=abc$ for some involutions $a,b,c$ in $\GL(W \oplus H_1)$.
Next, we can write $H_2=\underset{x \in X}{\bigoplus} P_x$ in which each $P_x$ is a $2$-dimensional linear subspace
of $V$. By Corollary \ref{iI2cor}, for each $x \in X$ we can find involutions $a_x$, $b_x$ and $c_x$ in $\GL(P_x)$ such that
$a_xb_xc_x=i \id_{P_x}$. Now, consider the endomorphism $a$ of $V$ whose restriction to
$W \oplus H_1$ is $a$ and whose restriction to $P_x$ is $a_x$ for all $x \in X$: this is obviously an involution.
Likewise, we define $\tilde{b}$ and $\tilde{c}$, and we obtain $iu=\tilde{a}\tilde{b}\tilde{c}$.

In a similar fashion, one deduces from Theorem \ref{skew1involution2unipotents} and Corollary \ref{iI2cor}
that $iu$ is the product of one involution and two unipotent endomorphisms of index $2$ (in any prescribed order).
\end{proof}

The proof of Proposition \ref{infinitedim-1} is an easy adaptation of the previous one, where instead of Theorems \ref{skew3involutions} and
\ref{skew1involution2unipotents}, one uses Theorem \ref{skew3unipotent}, and
instead of Corollary \ref{iI2cor} one uses Lemma \ref{-I23U2lemma}.

Finally, it can be proved that the above results yield all the elements of $\GP_f(V)$
that are the product of three involutions (respectively, of two involutions and a unipotent element of index $2$,
of one involution and two unipotent elements of index $2$, of three unipotent elements of index $2$) in the group $\GL(V)$.
This is however another story to be told.

In a further article, the above results will be used to complete the classification of the products of three involutions in $\GL(V)$,
as well as for the other three types of decompositions we have considered earlier.

\subsection{Strategy, and structure of the article}

Let us start from the problem of decomposing a matrix $A \in \GL_n(\F)$ into the product of three involutions.
Note that this problem is invariant under replacing $A$ with a similar matrix $B$, and that it amounts
to finding an involution $S$ such that $SA$ is the product of two involutions. The following
notion and notation will thus be very convenient:

\begin{Def}
Let $A,B$ be matrices of $\GL_n(\F)$.

We say that $A$ is \textbf{i-adjacent} to $B$ whenever there exists an involution $S \in \GL_n(\F)$ such that
$SA \simeq B$: then, we write $A \iadj B$.

We say that $A$ is \textbf{u-adjacent} to $B$ whenever there exists a $U_2$-matrix $U$ such that
$UA\simeq B$; then, we write $A \uadj B$.
\end{Def}

\begin{Rems}
\begin{enumerate}[(i)]
\item The inverse of an involution is itself. The inverse of a $U_2$-matrix is a $U_2$-matrix.
It follows that both relations $\iadj$ and $\uadj$ are symmetric.
\item If $A \iadj B$, $A' \simeq A$ and $B' \simeq B$, then $A' \iadj B'$.
\item If $A \iadj B$ and $A' \iadj B'$ then $A \oplus A' \iadj B \oplus B'$.
\item If $A \uadj B$, $A' \simeq A$ and $B' \simeq B$, then $A' \uadj B'$.
\item If $A \uadj B$ and $A' \uadj B'$ then $A \oplus A' \uadj B \oplus B'$.
\end{enumerate}
\end{Rems}

If $A$ is i-adjacent to the product of two involutions, then it is the product of three involutions.
If $A$ is u-adjacent to the product of two $U_2$-matrices, then it is the product of three $U_2$-matrices.
And so on. This suggests a basic strategy:
\begin{enumerate}[(1)]
\item Devise ways to construct suitable matrices that are i-adjacent (or u-adjacent) to a given matrix.
\item Recognize products of two involutions, and products of two $U_2$-matrices, from their Jordan canonical form (or their rational canonical form).
\end{enumerate}

Point (2) is settled: we have already recalled the characterizations in Theorem \ref{theo2}: yet they require
a bit of caution with respect to the products of two $U_2$-matrices, because of the possible eigenvalue $-1$
in the characteristic not $2$ case.

Most of our efforts, in the first half of this article, will be geared towards problem (1). There has already been some good work on the matter
in the literature (see e.g. \cite{Liu}): in particular, the fact that any invertible cyclic matrix is i-adjacent to
any cyclic matrix of the same size and with opposite determinant has been already recognized and used with success by other authors
\cite{Ballantine,Liu}. Our key contribution here is the generalization of this idea to the so-called \emph{well-partitioned} matrices that were
introduced in \cite{dSPidempotentLC}: in short, a well-partitioned matrix is a block-diagonal matrix $A \oplus B$ in which
the matrices $A$ and $B$ have coprime characteristic polynomials and are themselves direct sums of companion matrices, with
at most one block of size $1$ in each.
While they are not truly generalizations of cyclic matrices,
well-partitioned matrices are extremely convenient to solve our problem: indeed, with the exception of the matrices
with characteristic polynomial having a sole monic irreducible divisor, any matrix is similar
to the direct sum of a well-partitioned matrix and a diagonalisable matrix with at most two eigenvalues.
Hence, after we give general results on well-partitioned matrices,
the rest of our effort will focus on transforming matrices that are diagonalisable with two eigenvalues, and even
more specifically those in which the eigenvalues have the same multiplicity.

In a recent work \cite{dSP3square}, a similar strategy was used to prove that for any matrix $A \in \Mat_n(\F)$ with trace $0$, the augmented
matrix $A \oplus 0_n$ is the sum of three square-zero matrices. We will use similar ideas, but things tend to be substantially more
complicated in the present context. One part of the additional complexity comes from
the elements of finite order in the multiplicative group $\F^*$. The other major source of additional difficulty comes from the necessity, in the
study of the matrices that are i-adjacent or u-adjacent to a well-partitioned matrix, to recognize some that are cyclic:
this has lead us to identify a very large class of matrices that are cyclic but not in an obvious way: see Lemma \ref{BQClemma}.

\vskip 3mm
The remainder of the article is laid out as follows.

In Section \ref{cyclicSection}, we introduce some additional notation, we recall
some basic results on cyclic matrices, and we develop the groundwork for the next part. The key new concept in this section
is the notion of a block-quasi-companion matrix, to be used in Section \ref{wellpartSection}.

In Section \ref{wellpartSection}, we explore well-partitioned matrices:
we prove various decomposition theorems involving well-partitioned matrices (mostly variations of known results, but better suited
to the present study); we finish the section with the Adaptation Theorem, a major result on
matrices that are i-adjacent or u-adjacent to a well-partitioned matrix
(Theorem \ref{AdaptationTheorem}). We conclude the section by obtaining decomposition results for cyclic or well-partitioned matrices, as easy consequences
of the previous groundwork.

In Section \ref{Length4Section}, we prove Theorem \ref{theo4}.
The proof we will give is certainly not the shortest one in some cases, but it has the main upside of requiring
little discussion on the five types of decompositions! The strategy is simple: we start from a matrix $A \in \GL_n(\F)$
with determinant $\pm 1$. When $A$ is cyclic, the result is known (see Proposition \ref{cyclicdecomp}).
When $A$ is scalar, its diagonal entry has finite order:
we write $A$ as the product of two well-chosen diagonal matrices (whose diagonal elements form cycles or half-cycles)
that are the product of two $U_2$-matrices, or of two involutions, or of one involution and one $U_2$-matrix.
When $A$ is neither scalar nor cyclic, we prove that it is u-adjacent to a well-partitioned matrix,
and then we use the decomposition theorems of Section \ref{wellpartSection} for well-partitioned matrices.
Note that a more elementary strategy is possible in three situations: for products of four involutions, products of four $U_2$-matrices,
and products of two involutions and two $U_2$-matrices, one can prove that any non-scalar matrix is similar
to the product of a lower-triangular matrix with only $1$'s on the diagonal, and an upper-triangular matrix in which all the diagonal
entries equal $1$ with the possible exception of the last one (see \cite{Sourour}). Then, each such matrix is the product of two involutions, the first
one is the product of two $U_2$-matrices, and ditto for the second one if its determinant equals $1$.

The remaining sections deal with the proofs of Theorems \ref{theo3invol} to \ref{skew1involution2unipotents}.
We start by establishing results that are largely common to all four situations in the stable length $3$ problem:
in Section \ref{commonSection}, we first prove adjacency results for matrices of the form $\alpha I_n \oplus \beta I_n$
where $\alpha$ and $\beta$ are distinct nonzero scalars, and then we combine them with the Adaptation Theorem
to obtain decomposition results in specific situations when we have the direct sum of such a matrix with a well-partitioned matrix
(Section \ref{generaldecompsection}).

Then, we turn to the specific situations in the stable length $3$ problem. First, we deal with products of three $U_2$-matrices
(Section \ref{3U2Section}, where we successively prove Theorems \ref{theo3U2} and \ref{skew3unipotent}). Then, we deal with products of three involutions
(Section \ref{3invSection}). Products of one involution and two $U_2$-matrices are easily dealt with in Section \ref{1U22invSection}.
We finish, in Section \ref{1inv2U2Section}, with the most difficult situation: products of one involution and two $U_2$-matrices.
In those sections, the extension of $A$ to $A\oplus I_k$ is called \emph{natural}, whereas
the extension of $A$ to $A \oplus -I_k$ (in Theorem \ref{skew3unipotent}) or to $A \oplus i I_k$
(in Theorems \ref{skew3involutions} and \ref{skew1involution2unipotents})
is called \emph{unnatural}.

Given an integer $n \geq 2$, one could seek to find the least integer $k \geq 0$ for which, for every field $\F$ and every
matrix $A \in \GL_n(\F)$ with determinant $\pm 1$,
the augmented matrix $A \oplus I_k$ is the product of three involutions. It turns out that $n$ is not the right answer
but very close to it. In Section \ref{optimalitySection}, we will briefly discuss the corresponding problem
in Theorems \ref{theo3U2} to \ref{skew1involution2unipotents}. It turns out that the optimal augmentation size always corresponds to a special case when $A$
is scalar. Improving our theorems involves a hefty dose of additional technicalities, and proving the optimality of the improved
statements is a tedious task that requires a careful use of the classification of products of two quadratic matrices (see \cite{dSPprod2}).
Moreover, our primary motivation for the present study comes from the infinite-dimensional setting, in which the size of the augmentation is a non-issue.
Hence, in that ultimate section we will state the optimal results but we will offer no proof.

A final word about mixed decompositions: it is seen in Theorem \ref{theo2} that any product of two $U_2$-matrices is also the product of
two involutions. Hence, a matrix that is the product of one involution and two $U_2$-matrices is also the product of three involutions.
In particular, Theorem \ref{theo3invol} is a corollary of Theorem \ref{theo3mixed2}, and in Theorem \ref{theo4}
one could reduce the situation to only three problems (products of two involutions and two $U_2$-matrices, products of one involution and three $U_2$-matrices,
and products of four $U_2$-matrices). We have used this trick to shorten the proof of Theorem \ref{theo4}.
However, as far as the stable length $3$ problem is concerned, we have chosen not to start from the most difficult situation, for two reasons:
firstly, we suspect that most readers will be more interested in the decompositions into involutions only, and hence they will more
quickly grasp the techniques if we focus first on them; secondly, the proofs that involve the recognition of products of two $U_2$-matrices tend to be
substantially more technical, and as a consequence we prefer to save them for later parts of the article.

\section{On cyclic matrices}\label{cyclicSection}

\subsection{Additional notation}\label{additionalNotSection}

We denote by $\N$ the set of all natural numbers, i.e.\ non-negative integers, and by $\Z$ the set of all integers.

\vskip 2mm
Given a square matrix $M \in \Mat_n(\F)$, we denote its characteristic polynomial by $\chi_M(t):=\det(t I_n-M)$.

\vskip 2mm
Let $p(t)=t^n-\underset{k=0}{\overset{n-1}{\sum}}a_k\,t^k \in \F[t]$ be a monic polynomial with degree $n$.
We define its \textbf{trace} by $a_{n-1}$, denoted by $\tr p$, and its \textbf{norm} by $(-1)^{n-1} a_0$, denoted by
$N(p)$.
The \defterm{companion matrix} of $p$ is defined as
$$C\bigl(p(t)\bigr):=\begin{bmatrix}
0 &   & & (0) & a_0 \\
1 & 0 & &   & a_1 \\
0 & \ddots & \ddots & & \vdots \\
\vdots & \ddots & & 0 & a_{n-2} \\
(0) & \cdots & 0 &  1 & a_{n-1}
\end{bmatrix} \in \Mat_n(\F).$$
The characteristic polynomial of $C(p(t))$ is precisely $p(t)$, and so is its minimal polynomial.
Given $n \in \N^*$ and $\alpha \in \F$, we simply write
$$C_n(\alpha):=C\bigl((t-\alpha)^n\bigr),$$
and we note that this matrix is similar to a Jordan cell with size $n$ with respect to the eigenvalue $\alpha$.
\vskip 2mm
Let $A \in \Mat_n(\F)$ and $X \in \F^n$. We say that $X$ is \defterm{cyclic for $A$} whenever
$(A^kX)_{k \in \N}$ spans $\F^n$. This is equivalent to $(A^k X)_{0 \leq k <n}$ being a basis of $\F^n$, and in that
case $A$ is similar to the companion matrix of $\chi_A$. More precisely, we have
$$P^{-1}AP=C(\chi_A) \quad \text{for}\; P:=
\begin{bmatrix}
X & AX & \cdots & A^{n-1}X
\end{bmatrix}.$$
We note that if $A$ is invertible, then for $X \in \F^n$ to be a cyclic vector of $A$ it suffices that $ \Vect\{A^kX \mid k \in \Z\}=\F^n$:
indeed, as $V:=\Vect\{A^kX \mid k \in \N\}$ is finite-dimensional and stable under $A$, it is also stable under $A^{-1}$,
and hence it contains $A^kX$ for every negative integer $k$, yielding $ \Vect\{A^kX \mid k \in \Z\}=V$.

A \textbf{good cyclic} matrix is a matrix of the form
$$A=\begin{bmatrix}
a_{1,1} & a_{1,2}  & \cdots &  & a_{1,n} \\
1 & a_{2,2} & &   &  \\
0 & \ddots & \ddots & & \vdots \\
\vdots & \ddots & \ddots & a_{n-1,n-1} & a_{n-1,n} \\
(0) & \cdots & 0 &  1 & a_{n,n}
\end{bmatrix}$$
with no specific requirement on the $a_{i,j}$'s for $j \geq i$.
Classically, such a matrix is always cyclic: more precisely the first vector of the standard basis is a cyclic vector for it.

\vskip 2mm
Finally, we denote by $\Mat_{n,p}(\F)$ the vector space of all $n$ by $p$ matrices with entries in $\F$,
and in this space we consider the matrix units
$$H_{n,p}:=\begin{bmatrix}
0 & \cdots & 1 \\
\vdots & (0) & \vdots \\
0 & \cdots & 0
\end{bmatrix}, \quad
K_{n,p}:=
\begin{bmatrix}
0 & \cdots & 0 \\
\vdots & (0) & \vdots \\
0 & \cdots & 1
\end{bmatrix} \quad \text{and} \quad
L_{n,p}:=\begin{bmatrix}
0 & \cdots & 0  \\
\vdots & (0) & \vdots  \\
1 & \cdots & 0
\end{bmatrix}.
$$

\subsection{Basic lemmas}

The first lemma is folklore and is an easy consequence of Roth's theorem (see \cite{Roth}):

\begin{lemma}\label{Rothlemma}
Let $A \in \Mat_n(\F)$, $B \in \Mat_p(\F)$, and $C \in \Mat_{n,p}(\F)$.
Assume that $\chi_A$ and $\chi_B$ are coprime. Then,
$$\begin{bmatrix}
A & C \\
0 & B
\end{bmatrix} \simeq \begin{bmatrix}
A & 0 \\
0 & B
\end{bmatrix}.$$
\end{lemma}

The next lemma will be crucial:

\begin{lemma}[Polynomial fit lemma]\label{polyfitgen}
Let $A \in \Mat_n(\F)$ and $B \in \Mat_m(\F)$ be cyclic matrices, and
$p(t)$ be a monic polynomial of degree $n+m$ such that $N(p)=(\det A)\,(\det B)$. \\
Let $X \in \F^m$ be a cyclic vector for $B$, and $Y\in \F^n$ be a cyclic vector for $A^T$.
Then, there exists a matrix $D \in \Mat_{n,m}(\F)$ such that
$$\begin{vmatrix}
tI_n-A & -D \\
tXY^T & tI_m-B
\end{vmatrix}=p(t).$$
\end{lemma}

To prove this, our starting point is a result of similar flavor that was proved in \cite{dSPidempotentLC} (see lemma 11 there):

\begin{lemma}\label{polyfit1}
Let $A \in \Mat_n(\F)$ and $B \in \Mat_m(\F)$ be good cyclic matrices, and
$p(t)$ be a monic polynomial of degree $n+m$ such that $\tr p=\tr(A)+\tr(B)$. \\
Then, there exists a matrix $D \in \Mat_{n,m}(\F)$ such that
$$\begin{vmatrix}
tI_n-A & -D \\
-H_{m,n} & tI_m-B
\end{vmatrix}=p(t).$$
\end{lemma}

This allows us first to obtain a weaker version of Lemma \ref{polyfitgen}, in which the matrices $A$, $B$, $X$ and $Y$ are much more specific:

\begin{lemma}\label{polyfit2}
Let $A \in \Mat_n(\F)$ and $B \in \Mat_m(\F)$ be good cyclic matrices, and
$p(t)$ be a monic polynomial of degree $n+m$ such that $N(p)=(\det A)(\det B)$.
Then, there exists a matrix $D \in \Mat_{n,m}(\F)$ such that
$$\begin{vmatrix}
tI_n-A & -D \\
 tH_{m,n} & tI_m-B
\end{vmatrix}=p(t).$$
\end{lemma}

\begin{proof}[Proof of Lemma \ref{polyfit2}]
We write $p(t)=\chi_A(t)\,\chi_B(t)+t\, q(t)$ for some $q(t) \in \F[t]$ with degree at most $n+m-2$.
It follows that $\chi_A(t)\,\chi_B(t)-q(t)$ is monic with degree $n+m$ and trace $\tr A+\tr B$.
By Lemma \ref{polyfit1}, we can find a matrix
$D \in \Mat_{n,m}(\F)$ such that
$$\begin{vmatrix}
tI_n-A & -D \\
-H_{m,n} & tI_m-B
\end{vmatrix}=\chi_A(t)\,\chi_B(t)- q(t).$$
Denote by $R(t)$ the minor of the characteristic matrix $\begin{bmatrix}
tI_n-A & -D \\
0_{m,n} & tI_m-B
\end{bmatrix}$ in the entry $(n+1,n)$.
Using the linearity of the determinant with respect to the $(n+1)$-th row leads to the two formulas
$$\begin{vmatrix}
tI_n-A & -D \\
 tH_{m,n} & tI_m-B
\end{vmatrix}=\chi_A(t)\,\chi_B(t)-tR(t)$$
and
$$\begin{vmatrix}
tI_n-A & -D \\
-H_{m,n} & tI_m-B
\end{vmatrix}=\chi_A(t)\,\chi_B(t)+R(t).$$
The second result yields $R(t)=-q(t)$, and hence we deduce from the first one that
$$\begin{vmatrix}
tI_n-A & -D \\
tH_{m,n} & tI_m-B
\end{vmatrix}=\chi_A(t)\,\chi_B(t)+t\, q(t)=p(t).$$
\end{proof}

We are now ready to prove Lemma \ref{polyfitgen}.

\begin{proof}[Proof of Lemma \ref{polyfitgen}]
We will reduce the situation to the one covered by Lemma \ref{polyfit2}.
Set $Q:=\begin{bmatrix}
X & BX & \cdots & B^{m-1}X
\end{bmatrix}$ and
$P:=\begin{bmatrix}
(A^T)^{n-1} Y & \cdots & A^T Y & Y
\end{bmatrix}$.
Our assumptions show that $P$ and $Q$ are invertible and that
$Q^{-1}BQ=C(\chi_B)$. Moreover, in denoting by $(E_1,\dots,E_n)$ the standard basis of $\F^n$, we see that the last $n-2$ columns of the matrix
$P^{-1}A^T P$ are $E_1,\dots,E_{n-1}$. Thus,
$(P^{-1}A^T P)^T=P^T A (P^T)^{-1}$ is a good cyclic matrix (it is actually very close to a companion matrix, but instead of having potential nonzero entries
in the last column it has potential nonzero entries in the first row).

The matrix $R:=(P^T)^{-1} \oplus Q \in \Mat_{n+m}(\F)$ is then invertible, and one checks that, for all $D \in \Mat_{n,m}(\F)$,
$$R^{-1} \begin{bmatrix}
tI_n-A & -D \\
t XY^T & tI_m-B
\end{bmatrix} R=\begin{bmatrix}
t I_n-P^T A (P^T)^{-1} & -P^T D Q \\
t Q^{-1} XY^T (P^T)^{-1} & tI_m-Q^{-1} BQ
\end{bmatrix}.$$
Note that $Q^{-1} X Y^T (P^T)^{-1}=H_{m,n}$. Indeed, $Q^{-1}X$ is the first vector of the standard basis of $\F^m$,
and $P^{-1}Y$ is the last vector of the one of $\F^n$.
Hence,
$$R^{-1} \begin{bmatrix}
tI_n-A & -D \\
t XY^T & tI_m-B
\end{bmatrix} R=\begin{bmatrix}
t I_n-P^T A (P^T)^{-1} & -P^T D Q \\
t H_{m,n} & tI_m-Q^{-1} BQ
\end{bmatrix}.$$

Remember that $P^T A (P^T)^{-1}$ and $Q^{-1} BQ$ are good cyclic matrices. As they are similar to $A$ and $B$, respectively,
their respective characteristic polynomials are $\chi_A$ and $\chi_B$.
Hence, by Lemma \ref{polyfit2} there exists $D' \in \Mat_{n,m}(\F)$ such that
$$\begin{vmatrix}
t I_n-P^T A (P^T)^{-1} & -D' \\
t H_{m,n} & t I_m-Q^{-1} BQ
\end{vmatrix} = p(t).$$
Setting $D:=(P^T)^{-1} D' Q^{-1}$, we deduce that
$$\begin{vmatrix}
t I_n-A & -D \\
t XY^T & tI_m-B
\end{vmatrix}=\begin{vmatrix}
tI_n- P^T A (P^T)^{-1} & -D' \\
t H_{m,n} & t I_m-Q^{-1} BQ
\end{vmatrix}=p(t).$$
\end{proof}

\subsection{Block quasi-companion matrices}

\begin{Def}
A square matrix $M=(m_{i,j}) \in \Mat_n(\F)$ is called \textbf{quasi-companion} whenever
$m_{i,j}=0$ for all $(i,j)\in \lcro 1,n-1\rcro^2$ such that $i\neq j+1$, i.e.\ when $M$ has the following shape:
$$M=\begin{bmatrix}
0 &   & & (0) & ? \\
? & 0 & &   & ? \\
0 & \ddots & \ddots & & \vdots \\
\vdots & \ddots & ? & 0 & ? \\
? & \cdots & ? &  ? & ?
\end{bmatrix}.$$

A matrix $M \in \Mat_n(\F)$ is called \textbf{block quasi-companion} (in abbreviated form: \textbf{BQC})
whenever there are quasi-companion matrices $D_1,\dots,D_N$ with respective sizes $d_1,\dots,d_N$,
and nonzero scalars $\beta_1,\dots,\beta_{N-1}$
such that
$$M=\begin{bmatrix}
D_1 & ? & & & (?) \\
\beta_1 K_{d_2,d_1} & D_2 & & \\
0 & \beta_2 K_{d_3,d_2} & \ddots & & \\
\vdots & & \ddots & D_{N-1} & ? \\
(0) & \cdots & 0 & \beta_{N-1} K_{d_N,d_{N-1}} & D_N
\end{bmatrix},$$
where the question marks represent unspecified blocks.
In that case $(d_1,\dots,d_N)$ is called a \textbf{characteristic list} of $M$
(in general there can be several such lists attached to $M$).
\end{Def}

Now, we prove that any \emph{invertible} BQC matrix is cyclic. More precisely, we establish the following result:

\begin{lemma}\label{BQClemma}
Let $A \in \GL_n(\F)$ be an invertible BQC matrix, and $(d_1,\dots,d_N)$ be an associated characteristic list.
Then:
\begin{enumerate}[(a)]
\item The $d_1$-th vector of the standard basis of $\F^n$ is cyclic for $A$.
\item The $(n-d_N+1)$-th vector of the standard basis of $\F^n$ is cyclic for $A^T$.
\end{enumerate}
\end{lemma}

\begin{proof}
Throughout the proof, we denote by $(e_1,\dots,e_n)$ the standard basis of $\F^n$.
For $k \in \lcro 0,N\rcro$, set
$$a_k:=\sum_{i=1}^k d_i, \quad V_k:=\Vect(e_i)_{1 \leq i \leq a_k} \quad \text{and} \quad
V'_k:=\Vect(e_i)_{a_k<i \leq n.}$$

To prove point (a), we set $W:=\Vect(A^ke_{d_1})_{k \in \Z}$. Note that this subspace is obviously stable under both $A$ and $A^{-1}$.

We prove by induction that $V_k \subset W$ for all $k \in \lcro 0,N\rcro$.
This inclusion is trivial for $k=0$. Let $k \in \lcro 0,N-1\rcro$ be such that $V_k \subset W$.
First, we claim that $W$ contains $e_{a_{k+1}}$: if $k=0$ this comes from having $e_{d_1}$ in $W$;
otherwise we use the assumptions on $A$ to obtain $Ae_{a_k}=\lambda e_{a_{k+1}}$ mod $V_k$
for some $\lambda \in \F \setminus \{0\}$, which yields the claimed result since $W$
is stable under $A$ and includes $V_k$.
Next, for all $i \in \lcro 1,d_{k+1}-1\rcro$, we have $Ae_{a_k+i}=\lambda e_{a_k+i+1}$ mod $V_k+\Vect(e_{a_{k+1}})$ for some
$\lambda \in \F$, whence
$Ae_{a_k+i}=\lambda e_{a_k+i+1}$ mod $W$, and we deduce that $e_{a_k+i}=\lambda A^{-1}e_{a_k+i+1}$ mod $W$ because $W$ is stable under $A^{-1}$.
Hence, by downward induction we get that $e_{a_k+i} \in W$ for all $i \in \lcro 1,d_{k+1}\rcro$, and we conclude that
$V_{k+1} \subset W$.

Therefore, by induction $W$ includes $V_N=\F^n$, which completes the proof of point (a) (see the basic considerations in
Section \ref{additionalNotSection}).

\vskip 3mm
To prove point (b), we set $B:=A^T$ and $W'=\Vect(B^ke_{a_{N-1}+1})_{k \in \Z}$. Note again that $W'$ is stable under $B$ and $B^{-1}$.
We prove by downward induction that $V'_k \subset W'$ for all $k \in \lcro 0,N\rcro$.
This inclusion is trivial if $k=N$. Now, we let $k \in \lcro 0,N-1\rcro$ be such that $V'_{k+1} \subset W$, and we
prove that $V'_k \subset W$.

First, we claim that $e_{a_{k+1}} \in W'$. Indeed:
\begin{itemize}
\item if $k=N-1$ and $d_N=1$ then this is known because $W$ contains $e_{a_{N-1}+1}$;
\item if $k=N-1$ and $d_N>1$, then $B e_{a_{N-1}+1}= \lambda e_{a_N}$ for some $\lambda \in \F$, and
since $B$ is invertible we find $\lambda \neq 0$ and hence $e_{a_N} \in W'$;
\item if $k<N-1$ then we see that $B e_{a_{k+2}}=\lambda e_{a_{k+1}}$ mod $V'_{k+1}$ for some $\lambda \in \F^*$,
and hence $e_{a_{k+1}} \in W'$ because $W'$ includes $V'_{k+1}$, contains
in particular $e_{a_{k+2}}$, and is stable under $B$.
\end{itemize}
If $d_{k+1}=1$, then the above is enough to see that $W$ includes $V'_k$.
Now, assume that $d_{k+1}>1$.
We see that $Be_{a_k+1} =\lambda e_{a_{k+1}}$ mod $V'_{k+1}$ for some scalar $\lambda$. Since
$W'$ is stable under $B^{-1}$, contains $e_{a_{k+1}}$ and includes $V'_{k+1}$, this yields $e_{a_k+1} \in W'$.
Finally, for all $i \in \lcro 2,d_{k+1}-1\rcro$, we have $Be_{a_k+i}=\lambda e_{a_k+i-1}$ mod $\Vect(e_{a_{k+1}})+V'_{k+1}$
for some $\lambda \in \F$, whence
$Be_{a_k+i}=\lambda e_{a_k+i-1}$ mod $W'$. Using once more the fact that $W'$ is stable under $B^{-1}$, we obtain by
induction that $e_{a_k+i} \in W'$ for all $i \in \lcro 1,d_{k+1}-1\rcro$.
Hence we have shown that $V'_k \subset W'$.

Therefore, by downward induction we find $\F^n=V'_0 \subset W'$, which shows that $e_{a_{N-1}+1}$ is a cyclic vector for the invertible matrix $A^T$.
\end{proof}

\section{Well-partitioned matrices}\label{wellpartSection}

\subsection{Definition}

\begin{Def}
A square matrix $M$ is called \textbf{well-partitioned} if there are positive integers $r$ and $s$
and monic polynomials $p_1,\dots,p_r,q_1,\dots,q_s$ in $\F[t]$ such that:
\begin{enumerate}[(i)]
\item $M=C(p_1) \oplus \cdots \oplus C(p_r) \oplus C(q_1) \oplus \cdots \oplus C(q_s)$;
\item $\deg p_i \geq 2$ for all $i \in \lcro 2,r\rcro$;
\item $\deg q_j \geq 2$ for all $j \in \lcro 1,s-1\rcro$;
\item Each polynomial $p_i$ is coprime to each polynomial $q_j$.
\end{enumerate}
Note that the polynomials $p_1,\dots,p_r,q_1,\dots,q_s$ are then uniquely determined by $M$
(beware that in (i) we really require an \emph{equality} and not a mere similarity).

If in addition at most one of $p_1$ and $q_s$ has degree $1$, we say that $M$ is \textbf{very-well-partitioned}.
\end{Def}

\subsection{Reducing a square matrix with the help of a well-partitioned matrix}

Here, we prove the following results. They are variations of a lemma that was proved in \cite{dSP3square} (lemma 3.1 there).

\begin{lemma}\label{VWPlemma}
Let $M \in \Mat_n(\F)$. Assume that $M$ has at least $\frac{n}{2}$ Jordan cells of size $1$ for the eigenvalue $0$.
Then, there exist non-negative integers $p,q,r$ such that $p+q+r=n$, a matrix $N \in \Mat_p(\F)$ and a scalar $\alpha \in \F \setminus \{0\}$
such that
$$M\simeq N \oplus \alpha\,I_q\oplus 0_{r}, \quad r \geq q,$$
and either $N$ is void, or $N$ is nilpotent and $q=0$, or $N$ is very-well-partitioned.
\end{lemma}

\begin{lemma}\label{VWPlemmaunstable}
Let $M \in \Mat_n(\F)$. Assume that $M$ has at most one Jordan cell of size $1$ for each one of its eigenvalues in $\F$,
and that the characteristic polynomial of $M$ is not a power of some irreducible polynomial.
Then, $M$ is similar to a well-partitioned matrix.
\end{lemma}

We start with the proof of Lemma \ref{VWPlemmaunstable} as it is easier:

\begin{proof}[Proof of Lemma \ref{VWPlemmaunstable}]
Since the characteristic polynomial of $M$ is not a power of an irreducible polynomial,
we deduce from the primary canonical form that we can split
$M \simeq A \oplus B$ in which $A$ and $B$ are nonvoid square matrices with coprime characteristic polynomials.
We write the invariant factors of $A$ as $p_1,\dots p_a$ and the ones of $B$ as $q_1,\dots,q_b$.
There is at most one integer $k$ for which $p_k$ is constant, otherwise $M$ would have several Jordan cells
of size $1$ for one of its eigenvalues in $\F$. Likewise, there is at most one integer $k$ such that $q_k$ is nonconstant.
Hence, the matrix
$$M':=C(p_a) \oplus \cdots \oplus C(p_1) \oplus C(q_1) \oplus \cdots \oplus C(q_b)$$
is well-partitioned, and obviously $M \simeq M'$.
\end{proof}

\begin{proof}[Proof of Lemma \ref{VWPlemma}]
The proof strategy is similar to the previous one, only the details differ.
If $0$ is the sole eigenvalue of $M$ in an algebraic closure of $\F$, then we take $N:=M$, $q=r=0$ and $\alpha=1$.
Assume now that the contrary holds. Then, $M \simeq A \oplus B$
in which $A$ is nilpotent and $B$ is invertible, both of them nonvoid.
The assumptions on $M$ show that the size of $A$ is at least $\frac{n}{2}$, and hence the one of $B$ is at most $\frac{n}{2}\cdot$
Using the rational canonical form, we find that $A \simeq 0_m \oplus A'$, where
$A'$ is the direct sum of companion matrices associated with polynomials of the form $t^i$ with $i \geq 2$
(possibly $A'$ is void).
Note that $m$ is the number of Jordan cells of size $1$ for the eigenvalue $0$ of $M$, whence $m \geq \frac{n}{2}$.
Moreover, the rational canonical form of $B$ can be written $B \simeq B' \oplus \alpha I_q$, in which $B'$ is the direct sum of invertible companion matrices with size at least $2$,
and $\alpha$ is a nonzero scalar (possibly $q=0$ here, in which case we take $\alpha=1$).
\begin{itemize}
\item If $A'$ and $B'$ are both nonvoid, then $A' \oplus B'$ is very-well-partitioned and
$M \simeq (A' \oplus B') \oplus \alpha I_q \oplus 0_m$.
Note that $q \leq \frac{n}{2} \leq m$ in that case.
\item If $A'$ is void but $B'$ is not, then $0_1 \oplus B'$
is very-well-partitioned and
$M \simeq (0_1 \oplus B') \oplus \alpha I_q \oplus 0_{m-1}$.
Note that $q\leq \frac{n}{2}-2 \leq m-1$ in that case.
\item If $A'$ is nonvoid but $B'$ is void, then $q>0$, $A' \oplus \alpha I_1$ is very-well-partitioned,
$M \simeq (A' \oplus \alpha I_1) \oplus \alpha I_{q-1} \oplus 0_m$, and again $q-1 \leq m$.
\item If $A'$ and $B'$ are both void, then $M \simeq \alpha I_q \oplus 0_m$ with $q\leq \frac{n}{2} \leq m$, and
the first possible outcome is satisfied.
\end{itemize}
\end{proof}

Actually, we will not use Lemma \ref{VWPlemma} directly but in the form of the following corollary.
It is easily deduced from the standard observation that, for every monic polynomial $p(t)\in \F[t]$ with degree $k$,
and every $\beta \in \F$, the matrix $C(p(t))+\beta I_k$ is similar to $C(p(t-\beta))$.

\begin{prop}\label{VWPcor}
Let $M \in \GL_n(\F)$. Assume that, for some nonzero scalar $\beta$, $M$ has at
least $\frac{n}{2}$ Jordan cells of size $1$ for the eigenvalue $\beta$.
Then, there exist non-negative integers $p,q,r$ such that $p+q+r=n$, a matrix $N \in \GL_p(\F)$ and a scalar $\alpha \in \F \setminus \{\beta\}$
such that
$$M\simeq N \oplus \alpha\,I_q\oplus \beta I_{r}, \quad r \geq q,$$
and either $N$ is void, or $N-\beta I_p$ is nilpotent and $q=0$, or $N$ is very-well-partitioned.
\end{prop}

\subsection{Adjacency results for cyclic or well-partitioned matrices}

\begin{prop}\label{cyclicfit1}
Let $A \in \GL_n(\F)$ be an invertible cyclic matrix and $p$ be a monic polynomial of degree $n$ such that $N(p)=\pm \det A$.

If $n$ is odd or $N(p)=-\det A$ then $A$ is i-adjacent to $C(p)$.
\end{prop}

\begin{proof}
Assume first that $N(p)=-\det A$, and write $p=t^n-\underset{k=0}{\overset{n-1}{\sum}} b_k\, t^k$.
Without loss of generality, we can assume that $A=C(r)$ for some monic polynomial $r=t^n-\underset{k=0}{\overset{n-1}{\sum}} a_k\, t^k$.
Hence, $a_0=-b_0$. Note that $a_0 \neq 0$ since $A$ is invertible.
Define then $S=(s_{i,j}) \in \Mat_n(\F)$ as the matrix such that $s_{i,1}=\frac{b_{i-1}-a_{i-1}}{a_0}$ for all $i \in \lcro 2,n\rcro$,
$s_{1,1}=-1$, $s_{i,i}=1$ for all $i \in \lcro 2,n\rcro$, and all the other entries equal zero. Then,
it is easily seen that $S^2=I_n$ and that $S\,C(r)=C(p)$.

Assume now that $n$ is odd and $N(p)=\det A$. Set $q:=-p(-t)$, so that $N(q)=-\det A$.
Then, there is an involution $S$ such that $SA \simeq C(q)$.
Hence, $(-S)A \simeq -C(q) \simeq C(p)$.
\end{proof}

With a similar proof, we obtain the following result (in the definition of $S$ from the above proof, it suffices to replace
the entry at the $(1,1)$-spot with $1$).

\begin{prop}\label{cyclicfit2}
Let $A \in \GL_n(\F)$ be an invertible cyclic matrix and $p$ be a monic polynomial of degree $n$ such that $N(p)=\det A$.
Then, $A$ is u-adjacent to $C(p)$.
\end{prop}

Now, we arrive at the main key of the present study, that can be viewed as a variation of the above two results:

\begin{theo}[Adaptation Theorem]\label{AdaptationTheorem}
Let $M \in \GL_n(\F)$ be an invertible well-partitioned matrix.
\begin{enumerate}[(a)]
\item For every monic polynomial $r \in \F[t]$ with degree $n$ such that $N(r)=\det M$,
the matrix $M$ is u-adjacent to $C(r)$.
\item There exists $\eta \in \{1,-1\}$ such that, for every monic polynomial $r \in \F[t]$ with degree $n$ such that $N(r)=\eta \det M$,
the matrix $M$ is i-adjacent to $C(r)$.
\item If in addition $M$ is very-well-partitioned, then for every monic polynomial $r \in \F[t]$ with degree $n$ such that $N(r)=\pm \det M$,
the matrix $M$ is i-adjacent to $C(r)$.
\end{enumerate}
\end{theo}

\begin{proof}
Let $\varepsilon \in \{1,-1\}$.
Denote by $p_1,\dots,p_u,q_1,\dots,q_v$ the polynomials associated with the well-partitioned matrix $M$,
and by $n_1,\dots,n_u,m_1,\dots,m_v$ their respective degrees.
For $k \in \N^*$, set $U_k:=I_{k-1} \oplus (\varepsilon I_1)$.
Set
$$S:=\begin{bmatrix}
U_{n_1} & 0_{n_1 \times n_2} &  & & & & & (0) \\
L_{n_2,n_1} & U_{n_2} & \ddots & \\
(0) & \ddots & \ddots &   \\
 &  & L_{n_u,n_{u-1}} & U_{n_u} & 0_{n_u \times m_1} &  \\
 & & & L_{m_1,n_u} & U_{m_1} & 0_{m_1 \times m_2} & &  \\
\vdots & & & & L_{m_2,m_1} & U_{m_2} & \ddots &  \\
 & & & &   & \ddots & \ddots &  0_{m_{v-1} \times m_v} \\
(0) & & & \cdots &   & (0)  & L_{m_{v},m_{v-1}} & U_{m_v}
\end{bmatrix}.$$
Using the fact that $n_2>1,\dots,n_u>1,m_1>1,\dots,m_{v-1}>1$,
it is easily seen that $(S-I_n)(S-\varepsilon I_n)=0$.
Note that $\det S$ is a power of $\varepsilon$.
From now on, we let $r(t) \in \F[t]$ be an arbitrary monic polynomial with degree $n$ such that $N(r)=\alpha\det M$
for some $\alpha \in \{1,-1\}$.

Next, set $a=\underset{k=1}{\overset{u}{\sum}}n_k$ and $b=\underset{k=1}{\overset{v}{\sum}}m_k$, and let
$U \in \Mat_{a,b}(\F)$.
We can rewrite
$$M=\begin{bmatrix}
M_1 & 0_{a \times b} \\
0_{b \times a} & M_2
\end{bmatrix} \quad \text{and} \quad S=\begin{bmatrix}
S_1 & 0_{a \times b} \\
? & S_2
\end{bmatrix},$$
where $S_1,M_1$ belong to $\GL_a(\F)$, and $S_2,M_2$ belong to $\GL_b(\F)$.
Along the same format, set
$$A_U:=\begin{bmatrix}
M_1 & U \\
0_{b \times a} & M_2
\end{bmatrix}.$$
In order to conclude, it would suffice to prove that $U$ can be chosen so that
$$SA_U \simeq C(r).$$
Assume indeed that such a matrix $U$ exists. Lemma \ref{Rothlemma} shows that
$A_U=Q^{-1}MQ$ for some $Q \in \GL_n(\F)$. The matrix $\widetilde{S}:=QSQ^{-1}$ is then annihilated by $(t-1)(t-\varepsilon)$ and it satisfies
$$\widetilde{S}M=Q(SA_U)Q^{-1} \simeq C(r),$$
which will conclude the proof.

In order to obtain the claimed existence, we look more closely at $SA_U$.
Note first that $\det (SA_U)=\det S \det M$.
One computes that
$$S M=\begin{bmatrix}
S_1 M_1 & 0_{a \times b} \\
L & S_2 M_2
\end{bmatrix}$$
where $$L:=\begin{bmatrix}
0_{m_1 \times (a-n_u)} & -p_u(0)\,K_{m_1,n_u} \\
0_{(b-m_1) \times (a-n_u)} & 0_{(b-m_1) \times n_u}
\end{bmatrix}.$$
Moreover, one computes that both $S_1M_1$ and $S_2M_2$ are BQC matrices with respective characteristic lists
$(n_1,\dots,n_u)$ and $(m_1,\dots,m_v)$.

Finally, and this is crucial, one carefully checks that $SA_U$ is itself block-quasi-companion with characteristic list
$(n_1,\dots,n_u,m_1,\dots,m_v)$. Hence, by Lemma \ref{BQClemma} the invertible matrix $SA_U$ is cyclic.
In order to conclude, it suffices to prove that $U$ can be adjusted so that
the characteristic polynomial of $SA_U$ be $r(t)$.

We can split $S=N S'$ where
$$S':=\begin{bmatrix}
S_1 & 0_{a \times b} \\
0_{b \times a} & S_2
\end{bmatrix}$$
and $N$ is the transvection matrix that acts on rows by adding to the $(a+m_1)$-th row the product of $\lambda$ with the $(a-n_u+1)$-th
row for some fixed nonzero scalar $\lambda \in \F \setminus \{0\}$.
Denote by $X$ the $m_1$-th vector of the standard basis of $\F^b$, and by $Y$ the $(a-n_u+1)$-th vector of the one of $\F^a$.
Then,
$$\det (t I_{a+b}-S A_U)=\det(t N^{-1} -S'A_U)
=\begin{vmatrix}
t I_a-S_1M_1 & -S_1 U \\
t (-\lambda X Y^T)   &  t I_b-S_2M_2 \\
\end{vmatrix}.$$
By Lemma \ref{BQClemma}, $X$ is cyclic for $S_2M_2$, and hence so is $-\lambda X$, and $Y$ is cyclic for $(S_1M_1)^T$.
If $\alpha=\det S$, Lemma \ref{polyfitgen} yields a matrix $U' \in \Mat_{a,b}(\F)$ such that
$$\begin{vmatrix}
t I_a-S_1M_1 & U' \\
t (-\lambda X Y^T)   &  t I_b-S_2M_2 \\
\end{vmatrix}=r(t)$$
and hence the matrix $U:=-S_1^{-1}U'$ satisfies the required conditions.

Now, we can conclude.
\begin{itemize}
\item If $\alpha=1$, then we take $\varepsilon:=1$ and we obtain $M \uadj C(r)$.
\item If $\varepsilon=-1$ and $\alpha=\det S$, then we obtain $M \iadj C(r)$.
\item Assume finally that $M$ is very-well-partitioned, that $\varepsilon=-1$ and that $\alpha=-\det S$.
Then, we can do a simple modification in the matrix $S$ that leaves all the arguments
of the above proof intact but yields a new involution $S$ of $\GL_n(\F)$ such that $\det S=\alpha$:
if $n_1>1$, we can safely replace the $n_1$-th diagonal entry of $S$ with its opposite;
otherwise $m_v>1$ because $M$ is very-well-partitioned, and then we can safely replace the $(n-m_v+1)$-th diagonal entry of $S$
with its opposite.
\end{itemize}
Hence, points (a), (b) and (c) are proved.
\end{proof}

\subsection{Decomposition of cyclic or well-partitioned matrices}

We start with a result that is widely known in the case of products of three involutions.

\begin{prop}\label{cyclicdecomp}
Let $p\in \F[t]$ be a monic polynomial with norm $\pm 1$. Let $k \in \{0,1,2\}$.
Then, $C(p)$ is the product of $k$ unipotent matrices of index $2$ and $3-k$ involutions.

Moreover, if $p$ has norm $1$ then $C(p)$ is the product of three $U_2$-matrices.
\end{prop}

\begin{proof}
Denote by $d$ the degree of $p$.
By Lemma \ref{cyclicfit1}, $C(p)$ is u-adjacent to $C(q)$ where $q:=(t-1)^{d-1}(t-\lambda)$ for some $\lambda \in \{1,-1\}$.
By Theorem \ref{theo2}, $C(q)$ is the product of two involutions.
Hence, $C(p)$ is the product of one $U_2$-matrix and two involutions.

Likewise, $C(p)$ is i-adjacent to $C(r)$ where $r(t):=(t-1)^{d-1}(t-\mu)$ for some $\mu \in \{1,-1\}$, and hence
$C(p)$ is the product of three involutions.

If $d$ is even, $C(p)$ is u-adjacent to $C(q_1)$ or to $C(q_2)$, where
$q_1:=(t-1)^{d/2} (t+1)^{d/2}$ and $q_2:=(t-1)^{d/2+1} (t+1)^{d/2-1}$, and
both matrices $C(q_1)$ and $C(q_2)$ are the product of a $U_2$-matrix and an involution (by Theorem \ref{theo2mixedbis}).
If $d$ is odd, then $C(p)$ is u-adjacent to $C(r_1)$ or to $C(r_2)$, where
$r_1:=(t-1)^{(d-1)/2} (t+1)^{(d+1)/2}$ and $r_2:=(t-1)^{(d+1)/2} (t+1)^{(d-1)/2}$, and
again both matrices $C(r_1)$ and $C(r_2)$ are the product of a $U_2$-matrix and an involution.
Hence, $C(p)$ is the product of two $U_2$-matrices and an involution.

Assume finally that $p$ has norm $1$. By Proposition \ref{cyclicfit1}, $C(p)$ is u-adjacent to $C((t-1)^d)$,
a matrix which is the product of two $U_2$-matrices by Theorem \ref{theo2}. Hence, $C(p)$
is the product of three $U_2$-matrices.
\end{proof}

Using Proposition \ref{AdaptationTheorem} instead of Propositions \ref{cyclicfit1}
and \ref{cyclicfit2}, the same line of reasoning yields the following new result:

\begin{prop}\label{wellpartdecomp}
Let $A \in \GL_n(\F)$ be such that $\det A=\pm 1$.
Assume that $A$ is similar to a well-partitioned matrix.
Then, for all $k \in \{0,1,2\}$, the matrix $A$ is the product of $k$ unipotent matrices of index $2$
and $3-k$ involutions.
Moreover, if $\det A=1$ then $A$ is the product of three $U_2$-matrices.
\end{prop}

Combining this last result with Lemma \ref{VWPlemmaunstable} yields Theorem \ref{unstabletheo3}.

\section{The length $4$ problem in $\GL_n(\F)$}\label{Length4Section}

Here, we give a proof of Theorem \ref{theo4}.
This is done in three steps. First, we consider the case of scalar matrices (Section \ref{scalar4section}).
Then, we prove that any invertible matrix that is neither scalar nor cyclic is u-adjacent to
a well-partitioned matrix (Section \ref{adj4section}). We will complete the proof of Theorem \ref{theo4} by
using Propositions \ref{cyclicdecomp} and \ref{wellpartdecomp}.

\subsection{The case of scalar matrices}\label{scalar4section}

\begin{lemma}\label{scalar4lemma1}
Let $\alpha \in \F^*$, and $n \geq 1$ be an integer such that $\alpha^n=\pm 1$.
Then, the matrix $\alpha I_n$ is the product of four involutions, and it is also
the product of two involutions and two $U_2$-matrices.
\end{lemma}

\begin{proof}
Set $A:=\underset{k=0}{\overset{n-1}{\bigoplus}}\, C_1(\alpha^{2k})$.
Noting that $\alpha^{2n}=1$, we see that $A$ is similar to $A^{-1}$, and we deduce from Theorem \ref{theo2}
that $A^{-1}$ is the product of two involutions.
Likewise $\alpha A=\underset{k=0}{\overset{n-1}{\bigoplus}} C_1(\alpha^{2k+1})$
is similar to its inverse (note that $C_1(\alpha^{2k+1})=C_1(\alpha^{2n-2k-1})^{-1}$ for all $k \in \lcro 0,n-1\rcro$),
and hence it is the product of two involutions.
Hence, $\alpha I_n=(\alpha A)A^{-1}$ is the product of four involutions.

We also claim that one of the matrices $\alpha A$ and $A^{-1}$ is the product of two $U_2$-matrices.
This is immediate if $\F$ has characteristic $2$, and hence in the remainder of the proof we assume that the characteristic of $\F$ is not $2$.

By Theorem \ref{theo2}, it suffices to prove that $-1$ is not an eigenvalue of one of $\alpha A$ and $A^{-1}$.
Assume on the contrary that $-1$ is an eigenvalue of both. Then, $-1=\alpha^p=\alpha^q$ for some pair $(p,q)$ of integers, with $p$
odd and $q$ even. Thus $\alpha^{q-p}=1$ with $q-p$ odd, which yields that $\alpha$ has finite odd order and shows that $-1$ is not a power of $\alpha$!
This is a contradiction. Hence, one of the matrices $\alpha A$ and $A^{-1}$ is the product of two $U_2$-matrices,
and the other one is the product of two involutions. Hence, their product $\alpha I_n$ is the product of two $U_2$-matrices
and two involutions.
\end{proof}

\begin{lemma}\label{scalar4lemma2}
Let $\alpha \in \F^*$, and $n \geq 1$ be an integer such that $\alpha^n=1$.
Then, the matrix $\alpha I_n$ is the product of four $U_2$-matrices.
\end{lemma}

\begin{proof}
Because of Lemma \ref{scalar4lemma1}, we only consider the case when the characteristic of $\F$ is not $2$.

Assume first that $n$ is odd. Then, $\alpha$ has odd order and hence $-1 \not\in \langle \alpha \rangle$.
Then, we set $A:=\underset{k=0}{\overset{n-1}{\bigoplus}}\, C_1(\alpha^{2k})$.
With the same method as in the proof of Lemma \ref{scalar4lemma1}, we find that both $A^{-1}$ and $\alpha A$
are products of two $U_2$-matrices (using the fact that $-1$ is not a power of $\alpha$),
and we conclude that $\alpha I_n$ is the product of four $U_2$-matrices.

Assume now that $n$ is even, and write $n=2m$. Note that $\alpha^m=\pm 1$.
Then, we set
$$A:=\underset{k=0}{\overset{m-1}{\bigoplus}}\, C_2(\alpha^{2k}).$$
This time around, we see that both $A^{-1}$ and $\alpha A$
are products of two $U_2$-matrices (indeed, like in the proof of Lemma \ref{scalar4lemma1}, we see that both are similar to their inverse, and all the Jordan cells for the eigenvalue $-1$
have size $2$). Hence, $\alpha I_n$ is the product of four $U_2$-matrices.
\end{proof}

\begin{lemma}\label{scalar4lemma3}
Let $\alpha \in \F^*$ and $n \geq 1$ be an integer such that $\alpha^n=\pm 1$.
Then, $\alpha I_n$ is the product of one $U_2$-matrix and three involutions.
Moreover, it is the product of three $U_2$-matrices and one involution.
\end{lemma}

\begin{proof}
Due to Lemma \ref{scalar4lemma1}, we only consider the case when the characteristic of $\F$ is not $2$.
Moreover, by Theorem \ref{theo2}, it suffices to prove that $\alpha I_n$ is the
product of three $U_2$-matrices and one involution.

We split the discussion into two main cases.

\noindent \textbf{Case 1: $n$ is even.}

We write $n=2m$.
Set
$$B_1:=\underset{k=1}{\overset{m-1}{\bigoplus}}\, C_2\bigl((-\alpha^2)^{k}\bigr) \quad \text{and} \quad
B:=\underset{k=0}{\overset{m-1}{\bigoplus}}\, C_2\bigl((-\alpha^2)^{k}\bigr)=C_2(1) \oplus B_1,$$
so that
$$\alpha B \simeq \underset{k=0}{\overset{m-1}{\bigoplus}}\, C_2\bigl(\alpha (-\alpha^2)^{k}\bigr).$$
Note that both matrices $B^{-1}$ and $\alpha B$ only have Jordan cells of size $2$.
For every integer $k$, we see that
$$(-\alpha^2)^k (-\alpha^2)^{m-k}=(-\alpha^2)^m =(-1)^m \alpha^n$$
and
$$\alpha (-\alpha^2)^k \alpha (-\alpha^2)^{m-1-k}=(-1)^{m-1} \alpha^n.$$
\begin{itemize}
\item Assume first that $(-1)^m \alpha^n=1$. Then,
$B_1^{-1}$ is similar to its inverse and $\alpha B$ is similar to the opposite of its inverse.
Since both matrices only have Jordan cells of size $2$, we deduce from Theorems \ref{theo2} and
\ref{theo2mixed} that $B^{-1}$ is the product of two $U_2$-matrices
and that $\alpha B$ is the product of one $U_2$-matrix and one involution.
\item Assume next that $(-1)^m \alpha^n=-1$.  Then,
$B_1^{-1}$ is similar to the opposite of its inverse and $\alpha B$ is similar to its inverse.
This time around, we combine Theorems \ref{theo2mixed} and \ref{theo2mixedbis} to see that
$B^{-1}$ is the product of one $U_2$-matrix and one involution, whereas Theorem \ref{theo2}
shows that $\alpha B$ is the product of two $U_2$-matrices.
\end{itemize}
In any case $\alpha I_n=B^{-1}(\alpha B)$ is the product of three $U_2$-matrices and one involution.

\vskip 4mm
\noindent \textbf{Case 2: $n$ is odd.}

If $\alpha^n=-1$, we see that $(-\alpha)^n=1$. Moreover, if $-\alpha I_n$ is the product of three $U_2$-matrices
and one involution, then so is $\alpha I_n$. Hence, it suffices to deal with the case when $\alpha^n=1$.
In that case, we see that $\alpha$ has odd order, which we denote by $q$, and
$n$ is a multiple of $q$. Hence, it suffices to prove that $\alpha I_q$ is the product of three $U_2$-matrices
and one involution.

Set
$$A_1:=\underset{k=1}{\overset{q-1}{\bigoplus}}\, C_1\bigl((-\alpha^2)^{k}\bigr) \quad \text{and} \quad
A:=\underset{k=0}{\overset{q-1}{\bigoplus}}\, C_1\bigl((-\alpha^2)^{k}\bigr)=C_1(1)\oplus A_1,$$
so that
$$\alpha A \simeq \underset{k=0}{\overset{q-1}{\bigoplus}}\, C_1\bigl(\alpha (-\alpha^2)^{k}\bigr).$$
With the same line of reasoning as in the beginning of the proof, one sees that $\alpha A$
is similar to its inverse, whereas $A_1^{-1}$ is similar to the opposite of its inverse.
Moreover, we note that no eigenvalue of $A_1^{-1}$ is a square root of $-1$: indeed
otherwise there would be an integer $k$ such that $(-\alpha^2)^{2k}=-1$, whence $-1 =\alpha^{4k}$,
whereas $-1$ is not in the subgroup generated by $\alpha$ because the order of $\alpha$ is odd.
Hence, by Theorem \ref{theo2mixed} the matrix $A^{-1}$ is the product of a $U_2$-matrix and an involution.

Next, we claim that one of the matrices $\alpha A$ and $-\alpha A$ is the product of two $U_2$-matrices.
Assume that the contrary holds. Since $\alpha A$ and $-\alpha A$ are both similar to their inverse,
$-1$ must be an eigenvalue of both, yielding two elements $k,l$ of $\lcro 0,q-1\rcro$ such that
$\alpha (-\alpha^2)^k=-1=-\alpha (-\alpha^2)^l$. Then $\alpha^{2k+1}=(-1)^{k+1}$ and
$\alpha^{2l+1}=(-1)^l$. Since $\alpha$ has odd order, $-1$ is not a power of it and hence $k$ is odd and $l$ is even, whence
they are distinct and $\alpha^{2k+1}=\alpha^{2l+1}$. Then, $q$ divides $2(k-l)$, and hence it divides $k-l$, which is absurd
because $k,l$ are distinct elements of $\lcro 0,q-1\rcro$.

Therefore, one of $\alpha A$ and $-\alpha A$ is the product of two $U_2$-matrices.
Yet, both $A^{-1}$ and $-A^{-1}$ are products of one $U_2$-matrix and one involution (using once more the fact that the opposite
of an involution is an involution). Hence, by writing $\alpha I_q=(\alpha A)\,A^{-1}=(-\alpha A)\,(-A^{-1})$,
we conclude that $\alpha I_q$ is the product of three $U_2$-matrices and one involution.
\end{proof}

\subsection{Converting non-scalar matrices into well-partitioned matrices}\label{adj4section}

Our aim here is to prove the following result:

\begin{prop}\label{adj4prop}
Let $M \in \GL_n(\F)$ be a matrix that is neither scalar nor cyclic. Then, $M$ is u-adjacent to a well-partitioned matrix.
\end{prop}

With a similar method, one can prove that $M$ is also i-adjacent to a well-partitioned matrix, but we will not use this result.

We start with a basic result on polynomials:

\begin{lemma}\label{polynomialrootslemma}
Let $I$ be a finite subset of $\F^*$, and let $\lambda \in \F^*$. Let $n$ be an integer greater than $1$.
Then, there exists a monic polynomial $q$ of degree $n$ such that $N(q)=\lambda$ and $q$ has no root in $I$.
\end{lemma}

This result is deduced from the following one, which is folklore:

\begin{lemma}\label{coverlemma}
Let $\calF$ be a finite-dimensional affine space over $\F$, and $\calF_1,\dots,\calF_n$
be proper affine subspaces of $\calF$ (possibly void), with $n<|\F|$. Then, $\calF_1,\dots,\calF_n$ do not cover $\calF$.
\end{lemma}

\begin{proof}[Proof of Lemma \ref{coverlemma}]
The result is obvious if $\calF$ is void. Assume now that it is not.
We prove the result by induction on the dimension of $\calF$. If it is less than or equal to $1$,
then the result is obvious (the $\calF_i$'s being either void or singletons).
Assume now that the dimension of $\calF$ is at least $2$.
We choose an affine hyperplane $\calH$ of $\calF$ that includes $\calF_1$.
Assume first that some affine hyperplane $\calH'$ that is parallel to $\calH$ is included in none of the $\calF_i$'s.
Then, $\calF_1 \cap \calH',\dots,\calF_n \cap \calH'$ are proper affine subspaces of $\calH'$
and hence by induction they do not cover $\calH'$; hence, $\calF_1,\dots,\calF_n$ do not cover $\calF$.

If the converse holds every affine hyperplane $\calH'$ of $\calE$ that is parallel to $\calH$
is included in $\calF_i$ for some $i$, and then it equals $\calF_i$, which leads to
$n \geq |\F|$. This contradicts our assumptions.
\end{proof}

\begin{proof}[Proof of Lemma \ref{polynomialrootslemma}]
For each $\alpha \in I$, consider the nonconstant affine map
$$f_\alpha : (x_k)_{1 \leq k \leq n-1} \in \F^{n-1} \mapsto \biggl(\sum_{k=1}^{n-1} x_k\, \alpha^k\biggr) +\alpha^n+(-1)^n\lambda.$$
We note that $|I| <|\F|$ since $I \subset \F^*$. Hence, the proper affine subspaces
$f_\alpha^{-1}\{0\}$, for $\alpha \in I$, do not cover $\F^{n-1}$. This yields a list
$x \in  \F^{n-1}$ such that $f_\alpha(x) \neq 0$ for all $\alpha \in I$. Hence, the polynomial
$t^n+\underset{k=1}{\overset{n-1}{\sum}} x_k t^k+ (-1)^n\lambda$ has the required properties.
\end{proof}

We are now ready to prove Proposition \ref{adj4prop}.

\begin{proof}[Proof of Proposition \ref{adj4prop}]
Using the rational canonical form of $M$, we lose no generality in assuming that
$$M = C(p_1) \oplus \cdots \oplus C(p_r) \oplus \alpha I_s$$
where $p_1,\dots,p_r$ are polynomials, all with degree at least $2$ and such that $N(p_1) \neq 0$, \ldots,
$N(p_r) \neq 0$, $r \geq 1$, $\alpha \in \F \setminus \{0\}$, and potentially $s=0$.
Moreover if $s=0$ then $r \geq 2$ since $M$ is not cyclic.

Now, we split the discussion into two cases.

\noindent \textbf{Case 1:} $s>0$. \\
Set
$$B:=\begin{cases}
\underset{i=1}{\overset{s/2}{\bigoplus}}\, C_2(\alpha) & \text{if $s$ is even} \\
\biggl[\underset{i=1}{\overset{(s-1)/2}{\bigoplus}} C_2(\alpha)\biggr] \oplus C_1(\alpha)
& \text{if $s$ is odd.}
\end{cases}$$
In any case, noting that $\alpha\, C_2(1) \simeq C_2(\alpha)$,
we see that $\alpha I_s \uadj B$.

Then, by using Lemma \ref{polynomialrootslemma}, we find, for each $i \in \lcro 1,r\rcro$,
a monic polynomial $q_i$ such that $N(q_i)=N(p_i)$, $q_i(\alpha) \neq 0$ and $\deg(p_i)=\deg(q_i)$.
By Lemma \ref{cyclicfit2}, we see that $C(p_i) \uadj C(q_i)$ for all $i \in \lcro 1,r\rcro$.
Hence, $A:=\bigl[C(q_1) \oplus \cdots \oplus C(q_r)\bigr] \oplus B$ is well-partitioned and $M \uadj A$.

\noindent \textbf{Case 2:} $s=0$. \\
Then, $r>1$. Using Lemma \ref{polynomialrootslemma}, we find, for each $i \in \lcro 2,r\rcro$,
a monic polynomial $q_i$ such that $N(p_i)=N(q_i)$, $\deg(q_i)=\deg(p_i)$ and
$q_i$ has no root in $\{1,N(p_1)\}$.
Set $q_1:=(t-1)^{d-1}(t-N(p_1))$, where $d:=\deg p_1$.
Then, $q_1$ is coprime to $q_2,\dots,q_r$, and hence
$A:=C(q_1) \oplus \cdots \oplus C(q_r)$ is well-partitioned.
Yet, by Lemma \ref{cyclicfit2}, we have $C(p_i) \uadj C(q_i)$ for all $i \in \lcro 1,r\rcro$.
Hence, $M \uadj A$, which completes the proof.
\end{proof}

\subsection{Concluding the proof of Theorem \ref{theo4}}

We are now ready to complete the proof of Theorem \ref{theo4}.

Let $M \in \SL_n(\F)$. We prove that $M$ is the product of four $U_2$-matrices.
It is known by Lemma \ref{scalar4lemma2} if $M$ is scalar, and by Proposition \ref{cyclicdecomp} if $M$ is cyclic (because in that case
$M$ is the product of three $U_2$-matrices).
Assume now that $M$ is neither scalar nor cyclic.
Then, by Proposition \ref{adj4prop}, there is a well-partitioned matrix $A \in \GL_n(\F)$ such that $M \uadj A$. Hence, $\det A=\det M=1$.
Then, by Proposition \ref{wellpartdecomp},
$A$ is the product of three $U_2$-matrices, and hence $M$ is the product of four $U_2$-matrices.

Next, let $M \in \GL_n(\F)$ be such that $\det M=\pm 1$. Let $k \in \{1,2,3\}$. We wish to prove that $M$
is the product of $k$ unipotent matrices of index $2$
and $4-k$ involutions. Again, it is known by Lemmas \ref{scalar4lemma1} and \ref{scalar4lemma3} if $M$
is scalar, and by Lemma \ref{cyclicdecomp} if $M$ is cyclic (in that case $M$ is the product of $k-1$ unipotent matrices of index $2$
and $4-k$ involutions). Assume now that $M$ is neither scalar nor cyclic.
Then, by Proposition \ref{adj4prop}, $M \uadj A$ for some well-partitioned matrix $A \in \Mat_n(\F)$.
Hence, $\det A=\det M=\pm 1$.
Then, by Proposition \ref{wellpartdecomp}, $A$ is the product of $k-1$ unipotent matrices of index $2$ and $4-k$ involutions,
and hence $M$ is the product of $k$ unipotent matrices of index $2$ and $4-k$ involutions.

In particular, $M$ is the product of two $U_2$-matrices and two involutions, and we deduce from Theorem \ref{theo2} that it is also the product
of four involutions. Hence, Theorem \ref{theo4} is now proved.

\section{Common results for the stable length $3$ problem}\label{commonSection}

In the present section, we gather some technical results that are more or less common to all four cases in the stable length $3$ problem.
Most of our results are concerned with matrices of the form $\alpha I_n \oplus \beta I_n$ with distinct nonzero scalars $\alpha$ and $\beta$.

\subsection{Adjacency results on specific diagonal matrices}

\begin{lemma}\label{basicblocklemma}
Let $\alpha,\beta,\gamma,\delta$ be nonzero scalars, with $\alpha \neq \beta$, and let
$x \in \F \setminus \{0\}$. Set $\pi:=\alpha\beta\gamma\delta$.
Let $n$ be a positive integer.
Then, there is a matrix $S \in \GL_{2n}(\F)$ that is annihilated by the polynomial $(t-\gamma)(t-\delta)$ and such that
$$S\,(\alpha I_n \oplus \beta I_n) \simeq C\bigl((t-x)^n(t-\pi x^{-1})^n\bigr).$$
\end{lemma}

This lemma is a consequence of the following result, which was proved in \cite{dSPprod2}
(see lemma 4.5 there):

\begin{lemma}\label{blockmatrixlemma}
Let $r\in \F[t]$ be a monic polynomial with degree $n>0$, and $d$ be a nonzero scalar.
Let $N \in \Mat_n(\F)$ be cyclic with characteristic polynomial $r$.
Then,
$$\begin{bmatrix}
0_n & -d I_n \\
I_n & N
\end{bmatrix}\simeq C\bigl(t^n\,r(t+d t^{-1})\bigr).$$
\end{lemma}

\begin{proof}[Proof of Lemma \ref{basicblocklemma}]
We start from an arbitrary monic polynomial $r \in \F[t]$, which we will adjust afterwards.

Set
$$A:=\begin{bmatrix}
\gamma I_n & 0_n \\
\alpha^{-1} I_n & \delta I_n
\end{bmatrix} \quad \text{and} \quad
B:=\begin{bmatrix}
\alpha I_n & C(r) \\
0_n & \beta I_n
\end{bmatrix}.$$
Then,
$$AB=\begin{bmatrix}
\alpha \gamma I_n & \gamma\,C(r) \\
I_n & \delta\beta I_n+ \alpha^{-1} C(r)
\end{bmatrix}.$$
Taking $P:=\begin{bmatrix}
I_n & -\alpha\gamma I_n \\
0_n & I_n
\end{bmatrix}$, we deduce that
$$P (AB) P^{-1}=\begin{bmatrix}
0_n & -\pi I_n \\
I_n & \alpha^{-1}C(r)+(\delta \beta+\alpha \gamma) I_n
\end{bmatrix}.$$
Set now $s(t):=\bigl(t-(x+\pi x^{-1})\bigr)^n$, so that $t^ns(t+\pi t^{-1})=(t-x)^n(t-\pi x^{-1})^n$.
The matrix $\alpha(C(s)- (\delta \beta+\alpha \gamma) I_n)$ is obviously cyclic.
Hence, if we choose $r$ as its characteristic polynomial, we deduce from Lemma \ref{blockmatrixlemma} that
$$AB \simeq C\bigl((t-x)^n(t-\pi x^{-1})^n\bigr).$$
Next, it is easily checked that $(t-\alpha)(t-\beta)$ annihilates $B$, and $(t-\alpha)^n(t-\beta)^n$
is the characteristic polynomial of $B$. As $\alpha \neq \beta$, we deduce that $B$ is diagonalisable and
its eigenspaces have dimension $n$, whence $B=Q(\alpha I_n\oplus \beta I_n)Q^{-1}$ for some $Q \in \GL_{2n}(\F)$.
Finally, taking $S:=Q^{-1} A Q$, we obtain
$$S\,(\alpha I_n\oplus \beta I_n)=Q^{-1}(AB) Q \simeq AB \simeq C\bigl((t-x)^n(t-\pi x^{-1})^n\bigr).$$
The conclusion follows because one checks that the polynomial $(t-\gamma)(t-\delta)$
annihilates $A$ (and hence it also annihilates $S$).
\end{proof}

\begin{lemma}\label{C1basic}
Let $\alpha,\beta,\gamma,\delta,x$ be nonzero scalars, with $\alpha \neq \beta$.
Set $\pi:=\alpha\beta\gamma\delta$ and assume that $x^2 \neq \pi$.
Then, there is a matrix $S \in \GL_2(\F)$ that is annihilated by $(t-\gamma)(t-\delta)$ and such that
$$S\,(\alpha I_1 \oplus \beta I_1) \simeq C_1(x) \oplus C_1(\pi x^{-1}).$$
\end{lemma}

\begin{proof}
As $x^2 \neq \pi$ we have $x \neq \pi x^{-1}$ and hence $C\bigl((t-x)(t-\pi x^{-1})\bigr) \simeq C(t-x)\oplus C(t-\pi x^{-1})$.
Thus, the result follows from Lemma \ref{basicblocklemma} applied to $n=1$.
\end{proof}

\begin{lemma}\label{C2basic}
Let $\alpha,\beta,\gamma,\delta,x$ be nonzero scalars, with $\alpha \neq \beta$. Set $\pi:=\alpha\beta\gamma\delta$.
Then, there is a matrix $S \in \GL_4(\F)$ that is annihilated by $(t-\gamma)(t-\delta)$ and such that
$$S\,(\alpha I_2 \oplus \beta I_2) \simeq C_2(x) \oplus C_2(\pi x^{-1}).$$
\end{lemma}

\begin{proof}
As in the previous proof, if $x \neq \pi x^{-1}$ the result follows directly
from Lemma \ref{basicblocklemma}  applied to $n=2$.
Assume now that $x=\pi x^{-1}$. Then, Lemma \ref{basicblocklemma} yields a
matrix $S' \in \GL_2(\F)$ that is annihilated by $(t-\gamma)(t-\delta)$
and such that
$$S'\,(\alpha I_1\oplus \beta I_1) \simeq C\bigl((t-x)(t-\pi x^{-1})\bigr)=C\bigl((t-x)^2\bigr)
=C\bigl((t-\pi x^{-1})^2\bigr).$$
Hence
$$(S' \oplus S')\, (\alpha I_1\oplus \beta I_1 \oplus \alpha I_1 \oplus \beta I_1) \simeq
C_2(x) \oplus C_2(\pi x^{-1}).$$
We can find a permutation matrix $P \in \GL_4(\F)$ such that
$$\alpha I_1\oplus \beta I_1 \oplus \alpha I_1 \oplus \beta I_1=P (\alpha I_2 \oplus \beta I_2)P^{-1}.$$
Hence, the matrix $S:=P^{-1}(S' \oplus S') P$ is annihilated by $(t-\gamma)(t-\delta)$ and satisfies
$$S\,(\alpha I_2 \oplus \beta I_2) \simeq (S' \oplus S')\, (\alpha I_1\oplus \beta I_1 \oplus \alpha I_1 \oplus \beta I_1)
\simeq  C_2(x) \oplus C_2(\pi x^{-1}).$$
\end{proof}

\subsection{Cycles of cyclic matrices}

The following notation will be extremely useful in the remainder of the article:

\begin{Not}
Let $n$ be a positive integer, and let $\pi \in \F^*$ and $d \in \N^*$.
We set
$$\calC_{n,d}(\pi):=\underset{k=-(n-1)}{\overset{n}{\bigoplus}} C_d(\pi^k),$$
a matrix that is similar to
$$\underset{k=0}{\overset{n-1}{\bigoplus}} \bigl(C_d(\pi^{-k}) \oplus C_d(\pi^{k+1})\bigr).$$
\end{Not}

\begin{lemma}\label{C2cycle}
Let $n$ be a positive integer, and $\alpha,\beta,\gamma,\delta$ be nonzero scalars with $\alpha \neq \beta$.
Set $\pi:=\alpha\beta\gamma\delta$.
Then, there is a matrix $S \in \GL_{4n}(\F)$ that is annihilated by $(t-\gamma)(t-\delta)$ and such that
$$S\,(\alpha I_{2n}\oplus \beta I_{2n}) \simeq \calC_{n,2}(\pi).$$
\end{lemma}

\begin{proof}
Noting that $\alpha I_{2n}\oplus \beta I_{2n}$ is similar to the direct sum of
$n$ copies of $\alpha I_2 \oplus \beta I_2$, it suffices to apply Lemma \ref{C2basic}.
\end{proof}

\begin{lemma}\label{C1cycle}
Let $n$ be a positive integer, and $\alpha,\beta,\gamma,\delta$ be nonzero scalars, with $\alpha \neq \beta$.
Let $\varepsilon \in \{-1,1\}$.
Set $\pi:=\alpha\beta\gamma\delta$.
Assume that $\pi^{2k+1} \neq 1$ for all $k \in \lcro 0,n-1\rcro$.
Then, there is a matrix $S \in \GL_{2n}(\F)$ that is annihilated by $(t-\gamma)(t-\delta)$ and such that
$$S\,(\alpha I_n\oplus \beta I_n) \simeq \underset{k=0}{\overset{n-1}{\bigoplus}}\bigl(C_1(\varepsilon \pi^{-k}) \oplus C_1(\varepsilon \pi^{k+1})\bigr).$$
\end{lemma}

\begin{proof}
The proof is similar to the one of Lemma \ref{C2cycle}, however we use Lemma \ref{C1basic} this time around.
This works because our assumptions show that $\varepsilon \pi^{-k} \neq \varepsilon \pi^{k+1}$ for
all $k \in \lcro 0,n-1\rcro$.
\end{proof}

The next result is a consequence of the classification of products of two $U_2$-matrices:

\begin{lemma}\label{U2cycles}
Let $n$ be a positive integer, and $\pi$ be a nonzero scalar.
Then, $C_2(\pi^{-n})\oplus \calC_{n,2}(\pi)$
is the product of two $U_2$-matrices.

Moreover, if $\pi^{2n}=1$, then $\calC_{n,2}(\pi)$ is also the product of two $U_2$-matrices.
\end{lemma}

\begin{proof}
Reorganizing the terms, we find
$$C_2(\pi^{-n})\oplus  \calC_{n,2}(\pi)
\simeq C_2(1) \oplus \underbrace{\underset{k=1}{\overset{n}{\bigoplus}} \,\bigl(C_2(\pi^{-k}) \oplus C_2(\pi^{k})\bigr)}_M.$$
By Theorem \ref{theo2}, the matrix $M$ is the product of two $U_2$-matrices, and so is $C_2(1)$
(indeed, here all the Jordan cells have size $2$).

Assume now that $\pi^{2n}=1$. Then $\varepsilon:=\pi^n$ belongs to $\{1,-1\}$, and
we can reorganize
$$\calC_{n,2}(\pi) \simeq
C_2(1) \oplus C_2(\varepsilon) \oplus \underset{k=1}{\overset{n-1}{\bigoplus}}\,\bigl(C_2(\pi^{-k}) \oplus C_2(\pi^{k})\bigr).$$
The conclusion then follows again from Theorem \ref{theo2}.
\end{proof}

The following result is proved in a similar fashion, using the characterization of products of
two involutions instead of the one of products of two $U_2$-matrices:

\begin{lemma}\label{I2cycles}
Let $n$ be a positive integer, let $\pi$ be a nonzero scalar and let $\varepsilon \in \{-1,1\}$.
Then, for every positive integer $d$, the matrix
$$C_d(\varepsilon \pi^{-n}) \oplus \underset{k=0}{\overset{n-1}{\bigoplus}}\,\bigl(C_d(\varepsilon \pi^{-k}) \oplus C_d(\varepsilon \pi^{k+1})\bigr)$$
is the product of two involutions, and
if $\pi^n=\pm 1$, then $\calC_{n,d}(\pi)$ is also the product of two involutions.
\end{lemma}

\subsection{A general result on simple diagonal matrices}

\begin{lemma}\label{superdiagonalLemma}
Let $\alpha$ and $\beta$ be distinct nonzero scalars.
Let $\varepsilon \in \{1,-1\}$ be such that $(\varepsilon \alpha\beta)^p =1$.
Set $A:=\alpha I_p \oplus \beta I_p$.
\begin{enumerate}[(i)]
\item If $\varepsilon=1$ then $\alpha I_p \oplus \beta I_p$ is the product
of three $U_2$-matrices, and also of one $U_2$-matrix and two involutions.
\item If $\varepsilon=-1$ then $\alpha I_p \oplus \beta I_p$ is the product
of three involutions, and also of one involution and two $U_2$-matrices.
\end{enumerate}
\end{lemma}

\begin{proof}
Set $\pi:=\varepsilon\alpha\beta$ and $q:=\lfloor p/2\rfloor$.
Assume that $\varepsilon=1$ (respectively, $\varepsilon=-1$).
By Lemma \ref{C2cycle}, the matrix
$\alpha I_{2q}\oplus \beta I_{2q}$ is u-adjacent (respectively, i-adjacent) to
$\calC_{q,2}(\pi)$.

If $p$ is even then the last statement of Lemma \ref{U2cycles} shows that $\calC_{q,2}(\pi)$
is the product of two $U_2$-matrices.

Assume now that $p$ is odd, so that $\pi^{-q}=\pi^{q+1}$.
Then, $\alpha I_1 \oplus \beta I_1$ is u-adjacent (respectively, i-adjacent) to $C_2(\pi^{-q})$.
Hence, $\alpha I_p \oplus \beta I_p$ is u-adjacent (respectively, i-adjacent) to $B':=C_2(\pi^{-q}) \oplus \calC_{q,2}(\pi)$.
Once more, by Lemma \ref{U2cycles}, the matrix $B'$ is both the product of two involutions and the product of two $U_2$-matrices.

The conclusions follow.
\end{proof}

\subsection{More general decompositions}\label{generaldecompsection}

\begin{prop}\label{3U2decompProp}
Let $N \in \GL_n(\F)$ be an invertible well-partitioned matrix with $n \geq 3$.
Let $q$ be a positive integer and $\alpha,\beta$ be distinct nonzero scalars.
Set $M:=N \oplus \alpha I_q \oplus \beta I_q$ and assume that $\det M=1$.
Assume also that $(\alpha \beta)^k \neq 1$ for all $k \in \lcro 1,q\rcro$.
Then, $M$ is the product of three $U_2$-matrices.
\end{prop}

\begin{proof}
We will prove that $M$ is u-adjacent to the product of two $U_2$-matrices. To this end, we set $\pi:=\alpha \beta$.

\noindent \textbf{Case 1:}  $q$ is even. We write $q=2p$. \\
By Lemma \ref{C2cycle},
$$\alpha I_{2p} \oplus \beta I_{2p} \uadj \calC_{p,2}(\pi).$$
Note that $\pi^{-p} \neq 1$ and $\det N=\pi^{-2p}$.
Hence, the Adaptation Theorem yields
$$N \uadj C\bigl((t-1)^{n-2}(t-\pi^{-p})^2\bigr) \simeq C_{n-2}(1) \oplus C_2(\pi^{-p}),$$
whence
$$M \uadj  C_{n-2}(1) \oplus \bigl[C_2(\pi^{-p}) \oplus \calC_{p,2}(\pi)\bigr]$$
and the latter matrix is the product of two $U_2$-matrices by Lemma \ref{U2cycles} and Theorem \ref{theo2}.

\vskip 3mm
\noindent \textbf{Case 2:}  $q$ is odd. We write $q=2p+1$. \\
\noindent \textbf{Subcase 2.1:} $q>3$. \\
Then, $\det N=\pi^{-1}(\pi^{-p})^2$ and we note that the assumptions show that $\pi^{-1}$,
$\pi^{-p}$ and $1$ are pairwise distinct (indeed $p \geq 2$).
As $n \geq 3$, the Adaptation Theorem yields
$$N \uadj C\bigl((t-1)^{n-3}(t-\pi^{-1})(t-\pi^{-p})^2\bigr)
\simeq C_{n-3}(1) \oplus C_1(\pi^{-1}) \oplus C_2(\pi^{-p}).$$
On the other hand, $\alpha I_{q-1} \oplus \beta I_{q-1} \uadj \calC_{p,2}(\pi)$
and $\alpha I_1 \oplus \beta I_1 \uadj C_1(1) \oplus C_1(\pi)$.
Hence,
$$M \uadj M_1:= C_{n-3}(1) \oplus C_1(1) \oplus \bigl[C_1(\pi^{-1}) \oplus C_1(\pi)\bigr]
 \oplus \bigl[C_2(\pi^{-p}) \oplus \calC_{p,2}(\pi)\bigr].$$
The matrix $M_1$ is similar to its inverse. Moreover, the assumptions show that
$\pi^2 \neq 1$ (as $q \geq 2$) and hence $\pi$ is distinct from $-1$. It follows from Theorem \ref{theo2}
that $M_1$ is the product of two $U_2$-matrices.

\noindent \textbf{Subcase 2.2:} $q=3$. \\
In particular, the assumptions show that $\pi^2 \neq 1$.
By Lemma \ref{basicblocklemma},
$$\alpha I_q \oplus \beta I_q \uadj C\bigl((t-1)^3(t-\pi)^3\bigr)
\simeq C_3(1) \oplus C_3(\pi),$$
where the last similarity comes from having $\pi \neq 1$.
Note that $\det N=\pi^{-3}$ and $\pi^{-1} \neq 1$.
Hence, as $n \geq 3$ the Adaptation Theorem shows that
$$N \uadj C_{n-3}(1) \oplus C_3(\pi^{-1}).$$
It follows that
$$M \uadj C_{n-3}(1) \oplus C_3(1) \oplus C_3(\pi^{-1}) \oplus C_3(\pi).$$
We note that the latter matrix is similar to its inverse and $-1$ is no eigenvalue of it,
and we conclude that it is the product of two $U_2$-matrices.

\noindent \textbf{Subcase 2.3:} $q=1$ and $\pi \neq -1$. \\
Note that the assumptions show that $\pi \neq 1$. Moreover, $\det N=\pi^{-1}$.
The Adaptation Theorem shows that $N \uadj C_{n-1}(1) \oplus C_1(\pi^{-1})$,
whereas $\alpha I_q \oplus \beta I_q \uadj C_1(1) \oplus C_1(\pi)$.
Therefore,
$$M \uadj C_{n-1}(1) \oplus C_1(1) \oplus C_1(\pi) \oplus C_1(\pi^{-1}),$$
and as $\pi \neq -1$ the latter matrix is the product of two $U_2$-matrices.

\noindent \textbf{Subcase 2.4:} $q=1$ and $\pi=-1$. \\
Note that $\det N=-1$.
The matrix $\alpha I_1 \oplus \beta I_1$ is cyclic with characteristic polynomial $t^2-(\alpha+\beta)t-1$,
whence Proposition \ref{cyclicfit2} yields
$$\alpha I_1 \oplus \beta I_1 \uadj C(t^2+t-1).$$
Besides, since $1$ is not a root ot $t^2-t-1$, the Adaptation Theorem yields
$$N \uadj C\bigl((t-1)^{n-2}(t^2-t-1)\bigr) \simeq C_{n-2}(1) \oplus C(t^2-t-1).$$
Therefore,
$$M \uadj  C_{n-2}(1) \oplus C(t^2-t-1) \oplus C(t^2+t-1).$$
The latter matrix is similar to its inverse and $-1$ is no eigenvalue of it: hence it is the product of two $U_2$-matrices.

In any case, we have shown that $M$ is u-adjacent to the product of two $U_2$-matrices, and hence it is the product of three $U_2$-matrices.
\end{proof}

\begin{prop}\label{not3U2decompProp}
Assume that the characteristic of $\F$ is not $2$.
Let $N \in \GL_n(\F)$ be an invertible very-well-partitioned matrix.
Let $q$ be a positive integer and $\alpha,\beta$ be distinct nonzero scalars.
Set $M:=N \oplus \alpha I_q \oplus \beta I_q$ and assume that $\det M=\pm 1$.
Assume finally that $(\alpha \beta)^k \neq \pm 1$ for all $k \in  \lcro 1,q\rcro$.
Then, $M$ is the product of three involutions, but also of one involution and two $U_2$-matrices,
and also of one $U_2$-matrix and two involutions.
\end{prop}

\begin{proof}
Let $\varepsilon \in \{1,-1\}$, and define $\eta:=1$ if $\varepsilon=-1$, and $\eta:=\det M$ otherwise.
Set $\pi:=\varepsilon \alpha \beta$, and note that $\pi^k \neq \pm 1$ for all $k \in \lcro 1,q\rcro$.

\noindent \textbf{Case 1:} There is no integer $k \in \lcro 1,q\rcro$ for which $\pi^{2k+1}=1$.

Assume that $\varepsilon=-1$ (respectively, $\varepsilon=1$). Then, by Lemma \ref{C1cycle} the matrix
$\alpha I_q \oplus \beta I_q$ is i-adjacent (respectively, u-adjacent) to
$$\underset{k=0}{\overset{q-1}{\bigoplus}} \,\bigl(C_1(\eta \pi^{-k}) \oplus C_1(\eta \pi^{k+1})\bigr),$$
whereas the Adaptation Theorem shows that $N$ is i-adjacent (respectively, u-adjacent) to
$$C\bigl((t-1)^{n-1}(t-\eta \pi^{-q})\bigr) \simeq C_{n-1}(1) \oplus C_1(\eta \pi^{-q}).$$
Hence, $M$ is i-adjacent (respectively u-adjacent) to
$$M':=C_{n-1}(1) \oplus C_1(\eta \pi^{-q}) \oplus \underset{k=0}{\overset{q-1}{\bigoplus}}\,\bigl(C_1(\eta \pi^{-k}) \oplus C_1(\eta \pi^{k+1})\bigr).$$
The matrix $M'$ is obviously similar to its inverse and, if in addition $\eta=1$ then $-1$ is no eigenvalue of it.
Hence, $M'$ is the product of two involutions, and it also the product of two $U_2$-matrices if $\varepsilon=-1$.
This yields the claimed result for $M$.

\vskip 3mm
\noindent \textbf{Case 2:} There is an integer $k \in \lcro 0,q-1\rcro$ for which $\pi^{2k+1}=1$.
We take the least such integer $a$. Then, our starting assumptions show that $2a+1>q$, whence $a \geq \frac{q}{2}\cdot$
Set $b:=q-a$, so that $1 \leq b \leq a \leq q$.

\noindent \textbf{Subcase 2.1:} $b<a$ or $\eta=-1$. \\
Then, $\pi^a$ and $\eta \pi^b$ are distinct. Indeed, $\pi^a \neq \pi^b$ if $b<a$,
and on the other hand $-1$ does not belong to the group $\langle \pi\rangle$ because $\pi$ has odd order.
Assume that $\varepsilon=-1$ (respectively, $\varepsilon=1$). By Lemma \ref{C1cycle},
the matrix $\alpha I_a \oplus \beta I_a$ is i-adjacent (respectively, u-adjacent) to
$\calC_{a,1}(\pi)$, whereas $\alpha I_b \oplus \beta I_b$ is i-adjacent (respectively, u-adjacent) to
$$K:=\underset{k=0}{\overset{b-1}{\bigoplus}} \,\bigl(C_1(\eta \pi^{-k}) \oplus C_1(\eta \pi^{k+1})\bigr).$$

Finally, since $1,\pi^{-a},\eta \pi^{-b}$ are pairwise distinct, the Adaptation Theorem shows that $N$ is i-adjacent
(respectively, u-adjacent) to
$$C\bigl((t-1)^{n-2}(t-\pi^{-a})(t-\eta \pi^{-b})\bigr) \simeq C_{n-2}(1) \oplus C_1(\pi^{-a}) \oplus C_1(\eta \pi^{-b}).$$
Hence, $M$ is i-adjacent (respectively, u-adjacent) to
$$M':=C_{n-2}(1) \oplus \bigl[C_1(\pi^{-a}) \oplus \calC_{a,1}(\pi)\bigr] \oplus \bigl[C_1(\eta \pi^{-b}) \oplus K\bigr].$$

The matrix $M'$ is obviously similar to its inverse, and if $\eta=1$ then $-1$ is no eigenvalue of $M'$.
Hence, $M'$ is the product of two involutions, and it is also the product of two $U_2$-matrices if $\varepsilon=-1$.

\noindent \textbf{Subcase 2.2:} $b=a$ and $\eta=1$. \\
Hence, $\det N=\pm (\pi^{-a})^2$, and more precisely $\eta \det N=(\pi^{-a})^2$ if $\varepsilon=1$.
Assume that $\varepsilon=-1$ (respectively, $\varepsilon=1$).
Then, the Adaptation Theorem shows that $N$ is i-adjacent (respectively, u-adjacent) to
$$C\bigl((t-1)^{n-2}(t-\pi^{-a})^2\bigr) \simeq C_{n-2}(1) \oplus C_2(\pi^{-a}).$$
Moreover, since $q=2a$, we find that $\alpha I_q \oplus \beta I_q$ is i-adjacent (respectively, u-adjacent)
to $\calC_{a,2}(\pi)$.
Hence, $M$ is i-adjacent (respectively, u-adjacent) to
$$M':=C_{n-2}(1) \oplus \bigl[C_2(\pi^{-a}) \oplus \calC_{a,2}(\pi)\bigr],$$
a matrix which is the product of two $U_2$-matrices.

Hence, in any case we deduce that $M$ is the product of one involution and two $U_2$-matrices, as well as the product
of one $U_2$-matrix and two involutions. By the former, $M$ is also the product of three involutions.
\end{proof}

We finish with two variations of the previous two results that are relevant to unnatural extensions.

\begin{prop}\label{3U2decompPropskew}
Assume that the field $\F$ does not have characteristic $2$.
Let $p,q$ be integers with $p>0$ and $q\in \{p-1,p\}$. Let $\alpha \in \F \setminus \{0,-1\}$.
Let $N \in \GL_n(\F)$ be a well-partitioned matrix with $n \geq 3$. Set $M:=N \oplus \alpha I_p \oplus (-I_q)$.
Assume that $\det M=1$ and that there is no integer $k \in \lcro 1,q\rcro$ such that $\alpha^k=\pm 1$.
Then, $M$ is the product of three $U_2$-matrices.
\end{prop}

\begin{proof}
If $p=q$, the result follows directly from Proposition \ref{3U2decompProp}.
Hence, in the remainder of the proof we only consider the case when $q=p-1$.

Assume first that $\alpha=1$. Then $M \simeq (N \oplus (-I_q) \oplus I_q) \oplus I_1$.
Moreover, $\det(N \oplus (-I_q) \oplus I_q)=1$, and hence, either by Proposition \ref{3U2decompProp}
if $q>0$, or by Proposition \ref{wellpartdecomp} otherwise, $N \oplus (-I_q) \oplus I_q$
is the product of three $U_2$-matrices. Therefore, so is $M$.

In the remainder of the proof, we assume that $\alpha \neq 1$.

If $p=1$, we have $\det N=\alpha^{-1}$; then, as $\alpha^{-1} \neq 1$, the Adaptation Theorem shows that
$N \uadj C_{n-1}(1) \oplus C_1(\alpha^{-1})$; hence,
$M$ is u-adjacent to $C_{n-1}(1) \oplus C_1(\alpha^{-1}) \oplus C_1(\alpha)$,
which is the product of two $U_2$-matrices because $\alpha \neq -1$.

In the remainder of the proof, we assume further that $p>1$ (and hence $q>0$).
We set $\pi:=-\alpha$. We shall prove that $M$ is u-adjacent to a matrix that is the product of two $U_2$-matrices.
Note that $\pi^k \neq -1$ for all $k \in \lcro 1,q\rcro$.

\vskip 3mm
\noindent \textbf{Case 1:} There is no integer $k \in \lcro 0,q-1\rcro$ such that $\pi^{2k+1}=1$. \\
Then, by Lemma \ref{C1cycle},
$$\alpha I_p \oplus (-I_q) \uadj C_1(\alpha) \oplus \calC_{q,1}(\pi).$$
As $q>0$, we have $\pi^q \neq \alpha$ otherwise $\alpha^{q-1}=(-1)^q$, and then $q-1>0$ and we contradict our assumptions on $\alpha$.
Hence, $1$, $\alpha^{-1}$ and $\pi^{-q}$ are pairwise distinct.
Since $\det N=\alpha^{-1} \pi^{-q}$, the Adaptation Theorem shows that
$$N \uadj C\bigl((t-1)^{n-2}(t-\alpha^{-1})(t-\pi^{-q})\bigr) \simeq C_{n-2}(1) \oplus C_1(\alpha^{-1}) \oplus C_1(\pi^{-q}).$$
It follows that
$$M \uadj M':=C_{n-2}(1) \oplus \bigl[C_1(\alpha^{-1}) \oplus C_1(\alpha)\bigr] \oplus \bigl[C_1(\pi^{-q}) \oplus \calC_{q,1}(\pi)\bigr].$$
By Lemma \ref{I2cycles}, the matrix $M'$ is similar to its inverse. Moreover, as
$\alpha \neq -1$ and $\pi^k \neq -1$ for all $k \in \lcro -q,q\rcro$,
we see that $-1$ is no eigenvalue of $M'$. Hence, Theorem \ref{theo2} shows that $M'$ is the product of two $U_2$-matrices.

\vskip 3mm
\noindent \textbf{Case 2:} $q$ is even. \\
We write $q=2a$ for some integer $a$.
Then, by Lemma \ref{C2cycle},
$$\alpha I_q \oplus (-I_q) \uadj \calC_{a,2}(\pi).$$
Note that $\det N=\alpha^{-1} (\pi^{-a})^2$ and that $\pi^a \neq 1$ due to our assumptions.
If $\pi^a=\alpha$ then $\alpha^{a-1}=(-1)^a$, which yields $a-1=0$ (because $0 \leq a-1 \leq q$) and
we obtain a contradiction. Hence, $1,\alpha^{-1},\pi^{-a}$ are pairwise distinct, and we deduce from the Adaptation Theorem that
$$N \uadj C\bigl((t-1)^{n-3}(t-\alpha^{-1})(t-\pi^{-a})^2\bigr) \simeq C_{n-3}(1) \oplus C_1(\alpha^{-1}) \oplus C_2(\pi^{-a}).$$
Hence,
$$M \uadj  C_{n-3}(1) \oplus \bigl[C_1(\alpha^{-1}) \oplus C_1(\alpha)\bigr] \oplus \bigl[C_2(\pi^{-a}) \oplus \calC_{a,2}(\pi)\bigr].$$
Remembering that $\alpha \neq -1$, we see that the latter matrix is the product of two $U_2$-matrices.
\vskip 3mm
\noindent \textbf{Case 3:} $q$ is odd and there is an integer $k \in \lcro 0,q-1\rcro$ such that $\pi^{2k+1}=1$. \\
We take the least such integer $a$. Note that $2a+1>q$ due to our assumptions.
Hence, $a\geq \frac{q}{2}$. Setting $b:=q-a$, we deduce that $1 \leq b<a<q$ because $q$ is odd.
It ensues that $1$, $\alpha^{-1}$, $\pi^{-a}$ and $\pi^{-b}$ are pairwise distinct.
Note that $\det N=\alpha^{-1} \pi^{-q}=\alpha^{-1} \pi^{-a}\pi^{-b}$.
Thus, the Adaptation Theorem yields
$$N \uadj C\bigl((t-1)^{n-3}(t-\alpha^{-1})(t-\pi^{-a})(t-\pi^{-b})\bigr) \simeq
C_{n-3}(1) \oplus C_1(\alpha^{-1})\oplus  C_1(\pi^{-a}) \oplus C_1(\pi^{-b}).$$
On the other hand, we note that $\pi^{2k+1} \neq 1$ for all $k \in \lcro 0,a-1\rcro$, and hence Lemma
\ref{C1cycle} shows that
$$\alpha I_a \oplus (-I_a) \uadj \calC_{a,1}(\pi) \quad \text{and} \quad
\alpha I_b \oplus (-I_b) \uadj \calC_{b,1}(\pi).$$
Combining the above two adjacency results yields that $M$ is u-adjacent to
$$M':= C_{n-3}(1) \oplus \bigl[C_1(\alpha^{-1}) \oplus C_1(\alpha)\bigr] \oplus \bigl[C_1(\pi^{-a}) \oplus \calC_{a,1}(\pi)\bigr]
\oplus \bigl[C_1(\pi^{-b}) \oplus \calC_{b,1}(\pi)\bigr].$$
By Theorem \ref{theo2} and Lemma \ref{I2cycles}, the matrix $M'$ is similar to its inverse.
Moreover, $-1$ is not a power of $\pi$: indeed, as $\pi^{2a+1}=1$ we see that $\pi$ has odd order.
In addition $\alpha \neq -1$, and hence $-1$ is no eigenvalue of $M'$. Therefore, Theorem \ref{theo2}
yields that $M'$ is the product of two $U_2$-matrices.

Hence, in any case $M$ is u-adjacent to the product of two $U_2$-matrices, and we conclude that $M$ is the product of three $U_2$-matrices.
\end{proof}

\begin{prop}\label{3I2decomppropskew}
Assume that $\F$ has characteristic not $2$ and let $i \in \F$ satisfy $i^2=-1$.
Let $q$ be a positive integer, $N \in \GL_n(\F)$ be a very-well-partitioned invertible matrix,
and let $\alpha$ and $\beta$ be distinct nonzero scalars such that
$(\alpha\beta)^q \det N=\pm i$. Assume furthermore that $(\alpha \beta)^k \not\in \{\pm 1,\pm i\}$ for all
$k \in \lcro 1,q\rcro$. Then, $M:=N \oplus \alpha I_q \oplus \beta I_q \oplus i I_1$ is the
product of three involutions, and it is also the product of one involution and two $U_2$-matrices.
\end{prop}

\begin{proof}
Set $\pi:=-\alpha \beta$. We will prove that $M$ is i-adjacent to a matrix that is
the product of two $U_2$-matrices. This will yield the claimed results.

\noindent \textbf{Case 1:} There is no integer $k \in \lcro 0,q-1\rcro$ for
which $-i \pi^{-k} =i \pi^{k+1}$, i.e.\ $\pi^{2k+1}=-1$.

Then, by Lemma \ref{C1basic},
$$\alpha I_q \oplus \beta I_q \iadj \underset{k=0}{\overset{q-1}{\bigoplus}} \,\bigl(C_1(-i\pi^{-k}) \oplus C_1(i \pi^{k+1})\bigr).$$
Besides, since $-i \pi^{-q} \neq 1$, the Adaptation Theorem shows that
$$N \iadj C\bigl((t-1)^{n-1}(t+i \pi^{-q})\bigr) \simeq C_{n-1}(1) \oplus C_1(-i \pi^{-q}).$$
Hence,
$$M  \iadj M':=C_{n-1}(1) \oplus \bigl[C_1(i) \oplus C_1(-i)\bigr] \oplus
\underset{k=1}{\overset{q}{\bigoplus}} \bigl(C_1(-i\pi^{-k}) \oplus C_1(i \pi^{k})\bigr).$$
The matrix $M'$ is obviously similar to its inverse. Moreover, the assumptions show that $-1$ is no eigenvalue of $M'$.
Hence, $M'$ is the product of two $U_2$-matrices.

\noindent \textbf{Case 2:} There is an integer $k \in \lcro 0,q-1\rcro$ for
which $-i \pi^{-k} =i \pi^{k+1}$. \\
Let us take the least such integer $a$. Then, $1\leq a<q$.
Setting $b:=q-a$, we have $1\leq b<q$.
Note that $\pi^{2k+1} \neq -1$ for all $k \in \lcro 0,a-1\rcro$. Moreover, $\pi$ does not have odd order
because $-1 \in \langle \pi\rangle$, whence $\pi^{2k+1}\neq 1$ for all $k \in \lcro 0,b-1\rcro$.
Hence, it follows from Lemma \ref{C1basic} that
$$\alpha I_a \oplus \beta I_a \iadj \underset{k=0}{\overset{a-1}{\bigoplus}} \bigl(C_1(-i\pi^{-k}) \oplus C_1(i \pi^{k+1})\bigr)
\quad \text{and} \quad
\alpha I_b \oplus \beta I_b \iadj \calC_{b,1}(\pi).$$
Besides, $-i \pi^{-a}$, $\pi^{-b}$, and $1$ are pairwise distinct and hence
the Adaptation Theorem shows that $N$ is i-adjacent to
$$C\bigl((t-1)^{n-2}(t+i \pi^{-a})(t-\pi^{-b})\bigr) \simeq C_{n-2}(1) \oplus C_1(-i\pi^{-a})
\oplus C_1(\pi^{-b}).$$
We conclude that $M$ is i-adjacent to
\begin{multline*}
M':= C_{n-2}(1)\oplus C_1(i) \oplus
\biggl[C_1(-i\pi^{-a}) \oplus \underset{k=0}{\overset{a-1}{\bigoplus}} \bigl(C_1(-i\pi^{-k}) \oplus C_1(i \pi^{k+1})\bigr)\biggr] \\
\oplus
\bigl[C_1(\pi^{-b}) \oplus \calC_{b,1}(\pi)\bigr],
\end{multline*}
which is similar to
$$C_{n-2}(1)\oplus \bigl[C_1(i) \oplus C_1(-i)\bigr] \oplus
\biggl[\underset{k=1}{\overset{a}{\bigoplus}} \bigl(C_1(-i\pi^{-k}) \oplus C_1(i \pi^{k})\bigr)\biggr] \\
\oplus
\bigl[C_1(\pi^{-b}) \oplus \calC_{b,1}(\pi)\bigr].$$

Hence, $M'$ is similar to its inverse. We claim that $-1$ is no eigenvalue of it.
Indeed, the starting assumptions on $\pi$ show that $\pi^k\neq -1$ for all $k \in \lcro -b,b\rcro$,
and $\pi^k \neq \pm i$ for all $k \in \lcro -a,a\rcro$.

We conclude that $M$ is the product of one involution and two $U_2$-matrices, and by Theorem \ref{theo2} it is also the product
of three involutions.
\end{proof}

\section{Products of three unipotent matrices of index $2$}\label{3U2Section}

\subsection{Additional results on diagonal matrices}

\begin{lemma}\label{-I23U2lemma}
The matrix $-I_2$ is the product of three $U_2$-matrices.
\end{lemma}

\begin{proof}
Noting that $-C_2(1) \simeq C_2(-1)$ and that $C_2(1)$ is a $U_2$-matrix, we see that $-I_2$ is u-adjacent to $C_2(-1)$.
Besides, $C_2(-1)$ is the product of two $U_2$-matrices, by Theorem \ref{theo2}. This yields the claimed result.
\end{proof}

\begin{lemma}\label{lastU2lemma1}
Assume that $\F$ does not have characteristic $2$.
Let $\alpha$ belong to $\F \setminus \{0,-1\}$.
Let $n \in \N^*$ be an odd integer such that $\alpha^n=1$.
Then, $\alpha I_n \oplus (-I_{n-1})$ is the product of three $U_2$-matrices.
\end{lemma}

\begin{proof}
We note that $\alpha$ has odd order, and we denote by $m$ its order.
We start by proving that $\alpha I_m \oplus (-I_{m-1})$ is the product of three $U_2$-matrices.

Note that $-1$ has even order (because $\F$ does not have characteristic $2$), and hence
it cannot be a power of $\alpha$. It follows that $-\alpha$ has even order.
 Therefore, Lemma \ref{C1cycle} yields
$$\alpha I_{m-1} \oplus (-I_{m-1}) \uadj  \underset{k=0}{\overset{m-2}{\bigoplus}}\,\bigl(C_1((-\alpha)^{-k}) \oplus C_1((-\alpha)^{k+1})\bigr),$$
whence
$$\alpha I_m  \oplus (-I_{m-1}) \uadj B:=
C_1(\alpha) \oplus \underset{k=0}{\overset{m-2}{\bigoplus}}\,\bigl(C_1((-\alpha)^{-k}) \oplus C_1((-\alpha)^{k+1})\bigr).$$
Noting that $\alpha (-\alpha)^{m-1}=\alpha^m=1$, we obtain $B \simeq B^{-1}$.
Moreover, we claim that $-1$ is no eigenvalue of $B$. Indeed,
assume that $(-\alpha)^k =-1$ for some integer $k$ such that $|k|<m$.
Then, $\alpha^k =(-1)^{k+1}$, and since $-1$ is not a power of $\alpha$ this yields $\alpha^k=1$.
Then, $k=0$ because of the definition of $m$, which is absurd. Hence, $B$ is the product of two $U_2$-matrices,
and we conclude that $\alpha I_m  \oplus (-I_{m-1})$ is the product of three such matrices.

If $n=m$, we are done. Otherwise we write
$$\alpha I_n \oplus (-I_{n-1}) \simeq \bigl(\alpha I_m \oplus (-I_{m-1})\bigr) \oplus (\alpha I_{n-m} \oplus (-I_{n-m})\bigr).$$
Note that $(-\alpha)^{n-m}=1$. Therefore, by Lemma \ref{superdiagonalLemma}, $\alpha I_{n-m} \oplus (-I_{n-m})$ is the product of three $U_2$-matrices.
Therefore, so is $\alpha I_n \oplus (-I_{n-1})$.
\end{proof}

\begin{lemma}\label{lastU2lemma2}
Assume that $\F$ does not have characteristic $2$.
Let $\alpha \in \F \setminus \{0,-1\}$.
Let $n \in \N^*$ be an even integer such that $\alpha^n=-1$.
Then, $\alpha I_n \oplus (-I_{n-1})$ is the product of three $U_2$-matrices.
\end{lemma}

\begin{proof}
Since $n$ is even, one of the powers of $\alpha$ is a square root of $-1$, which has order $4$,
and hence the order of $\alpha$ is a multiple of $4$.

Denote by $m$ the order of $\alpha$, and write $m=4q$, so that $\alpha^{2q}=-1$.
Note that $4q \leq 2n$, whence $n \geq 2q$.
The element $-\alpha$ does not have odd order otherwise $\alpha^{2k}=1$ for some odd integer $k$,
and $4q$ would then divide $2k$!
Hence, with the same line of reasoning as in the previous lemma, we find
$$\alpha I_{2q}  \oplus (-I_{2q-1}) \uadj B:=C_1(\alpha)
\oplus \underset{k=0}{\overset{2q-2}{\bigoplus}}\,\bigl(C_1((-\alpha)^{-k}) \oplus C_1((-\alpha)^{k+1})\bigr),$$
and, as $\alpha (-\alpha)^{2q-1}=1$, we note that $B$ is similar to its inverse.
Assume now that $-1$ is an eigenvalue of $B$. Then, $(-\alpha)^k=-1$ for some integer $k$ such that $|k|<2q$,
and hence $\alpha^{2k}=1$ with $|2k|<m$. It ensues that $k=0$, which leads to a contradiction. Therefore, $-1$ is no eigenvalue of $B$,
and we conclude that $B$ is the product of two $U_2$-matrices.

From there, by splitting
$$\alpha I_n \oplus (-I_{n-1}) \simeq \bigl(\alpha I_m \oplus (-I_{m-1})\bigr) \oplus \bigl(\alpha I_{n-m} \oplus (-I_{n-m})\bigr),$$
one concludes with exactly the same line of reasoning as in the proof of Lemma \ref{lastU2lemma1}.
\end{proof}

\subsection{Natural extensions}

Here, we prove Theorem \ref{theo3U2}.
Let $A \in \SL_n(\F)$.
We wish to prove that the matrix $M:=A \oplus I_n$ is the product of three $U_2$-matrices.

We start by applying Proposition \ref{VWPcor}:
there exist non-negative integers $p,q,r$ such that $p+q+r=2n$, a matrix $N \in \GL_p(\F)$ and a scalar $\alpha \in \F \setminus \{1\}$
such that
$$M\simeq N \oplus \alpha\,I_{q}\oplus I_{r} \quad \text{and $q \geq r$},$$
and either $N$ is very-well-partitioned, or $N-I_p$ is nilpotent and $q=0$, or $N$ is void.
Since $I_{r-q}$ is the product of three $U_2$-matrices, it suffices to prove that
$M':=N \oplus \alpha I_q\oplus I_q$, which has determinant $1$, is the
product of three $U_2$-matrices.

If $q=0$, then $\det N=1$ and the result follows directly from Theorem \ref{theo2} if $N-I_p$ is nilpotent,
whereas it follows from Proposition \ref{wellpartdecomp} if $N$ is very-well-partitioned.

In the rest of the proof, we assume that $q>0$.

If $N$ is void then $\alpha^q=\det M'=1$ and the result follows directly from Lemma \ref{superdiagonalLemma}.

Assume finally that $N$ is very-well-partitioned and that $q>0$. Note that $\alpha \not\in \{0,1\}$ since $M'$ is invertible.
Note also that $p \geq 3$ because $N$ is very-well-partitioned. If there is no integer $k \in \lcro 1,q\rcro$ such that $\alpha^k=1$,
then the result follows directly from Proposition \ref{3U2decompProp}.
Assume finally that there is an integer $k \in \lcro 1,q\rcro$ such that $\alpha^k=1$, and denote by $a$ the greatest such integer.
Then, $\alpha^k \neq 1$ for all $k \in \lcro 1,q-a\rcro$.
We split
$$M' \simeq \underbrace{\bigl[N \oplus \alpha I_{q-a} \oplus I_{q-a}\bigr]}_{M'_1} \oplus \underbrace{\bigl[\alpha I_a \oplus I_a\bigr]}_{M'_2}$$
and we note that $\det M'_2=\alpha^a=1$ and hence $\det M'_1=1$.
By Proposition \ref{3U2decompProp} if $q-a>0$, and by Proposition \ref{wellpartdecomp} otherwise,
the matrix $M'_1$ is the product of three $U_2$-matrices; so is $M'_2$ by Lemma \ref{superdiagonalLemma}.
We conclude that $M'$ is the product of three $U_2$-matrices.

This completes the proof of Theorem \ref{theo3U2}.

\subsection{Unnatural extensions: additional results}\label{skewunipotent}

In this section and in the following one, we assume that the field $\F$ does not have characteristic $2$.
Here, we establish preliminary results for the proof of Theorem \ref{skew3unipotent} (which will be performed in
the next section).

\begin{lemma}\label{skewU3lemma1}
Let $N \in \Mat_n(\F)$ be nilpotent. Denote by $s$ the number of its Jordan cells of odd size.
Then, $(-I_n+N)\oplus (-I_s)$ is the product of three $U_2$-matrices.
\end{lemma}

\begin{proof}
For a scalar $\lambda$ and a positive integer $k$, we denote by $J_k(\lambda):=\lambda I_k+C(t^k)$
the (transposed) Jordan cell of size $k$ associated with the eigenvalue $\lambda$.

For every non-negative integer $k$, the matrix $J_{2k}(-1)$ is the product of two $U_2$-matrices
(by Theorem \ref{theo2}), and hence it is also the product of three such matrices.

In order to conclude, it suffices to prove that for every non-negative integer $k$, the matrix
$(-I_1) \oplus J_{2k+1}(-1)$ is the product of three $U_2$-matrices.
Let $k$ be such an integer, and denote by $U$ the transvection matrix of $\GL_{2k+2}(\F)$ with entry $-1$ at the $(2,1)$-spot. One checks that
$U\,((-I_1) \oplus J_{2k+1}(-1))=J_{2k+2}(-1)$, and the latter matrix is the product of two $U_2$-matrices.
Noting that $U$ is a $U_2$-matrix, we conclude that $(-I_1) \oplus J_{2k+1}(-1)$ is the product of three $U_2$-matrices.
\end{proof}

\begin{lemma}\label{skewU3lemma2}
Let $A \in \GL_n(\F)$ be such that $\det A=\pm 1$, and $k$ be a positive integer such that $(-1)^k=\det A$.
Assume that $A$ is similar to a direct sum of companion matrices, all with size greater than $1$, and that
$-1$ is no eigenvalue of $A$. Then, $A \oplus (-I_k)$ is the product of three $U_2$-matrices.
\end{lemma}

\begin{proof}
Since $-I_2$ is the product of three $U_2$-matrices, it suffices to consider the case when $k \in \{1,2\}$.

If $k=1$, we see that $A \oplus (-I_1)$ is similar to a well-partitioned matrix, and we deduce from
Proposition \ref{wellpartdecomp} that $A \oplus (-I_1)$ is the product of three $U_2$-matrices.

Assume now that $k=2$, so that $\det A=1$. If $A$ is cyclic then Proposition \ref{cyclicdecomp} shows that it is the product of three $U_2$-matrices.
Since so is $-I_2$ (see Lemma \ref{-I23U2lemma}), so is $A \oplus (-I_2)$.
Assume finally that $A$ is non-cyclic. The assumptions allow us to split $A \simeq B_1 \oplus \cdots \oplus B_p$,
where $B_1,\dots,B_p$ are companion matrices with size at least $2$, and $p \geq 2$. If all the $B_i$'s have determinant $1$, then
they are all products of three $U_2$-matrices, and hence so is $A$. Assuming otherwise, we lose no generality in further assuming that
$\det B_1 \neq 1$. Then, we set $A_1:=(-I_1) \oplus B_1$ and $A_2:=(B_2 \oplus \cdots \oplus B_p) \oplus (-I_1)$
and we note that $\det A_1=-\det B_1 \neq -1$ and $\det A_2=(\det A_1)^{-1}$. Set $\alpha:=\det A_1$.
We also note that $A_1$ and $A_2$ are both well-partitioned. The Adaptation Theorem
shows that $A_1 \uadj C\bigl((t-1)^{s}(t-\alpha)\bigr)$ and $A_2 \uadj C\bigl((t-1)^{r}(t-\alpha^{-1})\bigr)$
for some positive integers $r$ and $s$.
Then, $A \oplus (-I_k)$ is u-adjacent to $C\bigl((t-1)^{s}(t-\alpha)\bigr) \oplus C\bigl((t-1)^{r}(t-\alpha^{-1})\bigr)$, a matrix
that is similar to its inverse and of which $-1$ is no eigenvalue. Hence
$A \oplus (-I_k)$ is the product of three $U_2$-matrices.
\end{proof}

\begin{lemma}\label{skewU3lemma3}
Let $A \in \GL_n(\F)$ be a very-well-partitioned matrix such that $\det A=-1$.
Then, $A \oplus (-I_1)$ is the product of three $U_2$-matrices.
\end{lemma}

\begin{proof}
Set $M:=A \oplus (-I_1)$.
We denote by $p_1,\dots,p_r,q_1,\dots,q_s$ the polynomials that are attached to $A$ as a well-partitioned matrix.
Without loss of generality, we can assume that $p_1,\dots,p_r$ are all coprime with $t+1$.
If $\deg(q_s)>1$, then $M$ is well-partitioned with determinant $1$ and we deduce from Proposition \ref{wellpartdecomp}
that it is the product of three $U_2$-matrices. Assume now that $\deg(q_s)=1$ and $q_s \neq t+1$.
Then, $C(q_s) \oplus (-I_1) \simeq C\bigl((t+1)q_s\bigr)$ and hence $M$ is similar to a well-partitioned matrix
(with attached polynomials $p_1,\dots,p_r,q_1,\dots,q_{s-1},(t+1)q_s$). Again, $M$ is the product of three $U_2$-matrices in that case.

Assume that $q_s=t+1$ and $s>1$. As $A$ is very-well-partitioned, the matrix $B:=C(p_1) \oplus \cdots C(p_r) \oplus C(q_1) \oplus \cdots \oplus
C(q_{s-1})$ is well-partitioned with determinant $1$. Hence, $B$ is the product of three $U_2$-matrices, and
so is $M=B \oplus (-I_2)$.

Assume finally that $q_s=t+1$ and $s=1$. Then, $B:=C(p_1) \oplus \cdots \oplus C(p_r)$ is the direct sum of companion matrices with size at least $2$
and $-1$ is no eigenvalue of $B$. We deduce from Lemma \ref{skewU3lemma2} that $M=B \oplus (-I_2)$ is the product of three $U_2$-matrices.
\end{proof}

\subsection{Unnatural extensions: proof of Theorem \ref{theo3U2}}

Here, we complete the proof of Theorem \ref{theo3U2}. We assume that $\F$ does not have characteristic $2$.
Let $A \in \GL_n(\F)$ be such that $\det A=\pm 1$, and let $m \geq n$ be an integer such that
$(-1)^m \det A=1$. We wish to prove that $M:=A \oplus (-I_m)$ is the product of three $U_2$-matrices.

By Proposition \ref{VWPcor}, there are non-negative integers $p,q,r$, a matrix $N \in \GL_p(\F)$,
and a scalar $\alpha \in \F \setminus \{-1\}$ such that
$$M \simeq N \oplus \alpha I_q \oplus (-I_r) \quad \text{and} \quad r \geq q,$$
and either $N+I_p$ is nilpotent and $q=0$, or $N$ is void, or
$N$ is very-well-partitioned.
Moreover, when $q=0$ and $N+I_p$ is nilpotent, we can assume that $N$ has no Jordan cell of size $1$
(otherwise we put all those cells in the last $-I_r$ block).

Assume first that $N$ is void. If $q=0$, then $r$ is even and it follows directly from Lemma \ref{-I23U2lemma}
that $M$ is the product of three $U_2$-matrices. Assume now that $q>0$. If $r-q$ is even, we
write that $M$ is similar to the direct sum of $\alpha I_q \oplus (-I_q)$ and of copies of $-I_2$,
and we conclude by combining Lemmas \ref{superdiagonalLemma} and \ref{-I23U2lemma}.
If $r-q$ is odd, we write that $M$ is similar to the direct sum of $\alpha I_q \oplus (-I_{q-1})$ and of copies of $-I_2$,
and we conclude by combining Lemma \ref{-I23U2lemma} with one of Lemmas \ref{lastU2lemma1} and \ref{lastU2lemma2}.

Assume now that $q=0$ and that $N+I_p$ is nilpotent. Then, $r$ is greater than or equal to the number $s$ of Jordan cells of odd
size of $N$, and $r-s$ is even because $\det M=1$. It follows from Lemmas \ref{-I23U2lemma} and \ref{skewU3lemma1} that $M$ is the product of three $U_2$-matrices.

Assume finally that $N$ is very-well-partitioned. If $q=0$ and $r$ is even, then $\det N=1$
and we combine Proposition \ref{wellpartdecomp} with Lemma \ref{-I23U2lemma} to obtain that $M$ is the product of three $U_2$-matrices.
If $q=0$ and $r$ is odd, the same conclusion is reached by combining Lemmas \ref{skewU3lemma3} and \ref{-I23U2lemma}.
In the remainder of the proof, we assume that $q>0$. Using Lemma \ref{-I23U2lemma} once more, we
choose $m \in \{q-1,q\}$ that equals $r$ modulo $2$, and we find that it suffices to prove that
$$M':=N \oplus \alpha I_q \oplus (-I_m),$$
which has determinant $1$, is the product of three $U_2$-matrices.

If $\alpha^k \neq \pm 1$ for all $k \in \lcro 1,m\rcro$, then Proposition \ref{3U2decompPropskew}
directly yields that $M$ is the product of three $U_2$-matrices.
Assume now that $\alpha^k=\pm 1$ for some $k \in \lcro 1,m\rcro$, and denote by $a$ the greatest such integer.
Hence, $\alpha^k \neq \pm 1$ for all $k \in \lcro 1,m-a\rcro$.

\begin{itemize}
\item Assume first that $(-\alpha)^a=1$. Then, we resplit
$$M' \simeq \underbrace{\bigl[N \oplus \alpha I_{q-a} \oplus (-I_{m-a})\bigr]}_{M_1} \oplus \underbrace{\bigl[\alpha I_a \oplus (-I_a)\bigr]}_{M_2}$$
and we note that $\det M_1=1$.
If $q-a>0$ then Proposition \ref{3U2decompPropskew} shows that
$M_1$ is the product of three $U_2$-matrices. If $q-a=0$ then Proposition \ref{wellpartdecomp} shows that
$M_1$ is the product of three $U_2$-matrices. Moreover, $M_2$ is also the product of three $U_2$-matrices, by Lemma
\ref{superdiagonalLemma}. Hence, so is $M'$.

\item If $(-\alpha)^a=-1$ and $m=q-1$, then we resplit
$$M' \simeq \bigl[N \oplus \alpha I_{q-a} \oplus (-I_{q-a})\bigr] \oplus \bigl[\alpha I_a \oplus (-I_{a-1})\bigr],$$
and this time around we conclude by combining one of Propositions \ref{wellpartdecomp} and \ref{3U2decompProp} with one of
Lemmas \ref{lastU2lemma1} and \ref{lastU2lemma2}.

\item If $(-\alpha)^a=-1$, $m=q$ and $q>a$, then we resplit
$$M' \simeq (-I_2) \oplus
\bigl[N \oplus \alpha I_{q-a} \oplus (-I_{q-a-1})\bigr] \oplus \bigl[\alpha I_a \oplus (-I_{a-1})\bigr],$$
and we conclude as in the preceding case, with Proposition \ref{3U2decompPropskew} instead of Proposition \ref{3U2decompProp},
and by using Lemma \ref{-I23U2lemma}.

\item Assume finally that $(-\alpha)^a=-1$ and $m=q=a$. Then, we split
$$M' \simeq
\bigl[N \oplus (-I_1)\bigr] \oplus \bigl[\alpha I_a \oplus (-I_{a-1})\bigr],$$
and we combine one of Lemmas  \ref{lastU2lemma1} and \ref{lastU2lemma2} with Lemma \ref{skewU3lemma3}
to conclude that $M'$ is the product of three $U_2$-matrices.
\end{itemize}

The proof of Theorem \ref{theo3U2} is now complete.

\section{Products of three involutions}\label{3invSection}

If $\F$ has characteristic $2$ then the involutions in $\GL_n(\F)$ are the $U_2$-matrices,
and our results are just consequences of Theorem \ref{theo3U2}. Hence, in the present section (and also in the following two), we assume that
the characteristic of $\F$ is not $2$.

\subsection{Natural extensions}

We start with an additional preliminary lemma:

\begin{lemma}\label{diagonal3invollemmafinal}
Let $\alpha$ and $\beta$ be distinct nonzero scalars, and let
$q$ be a positive integer such that $(-\alpha\beta)^q =\pm 1$. Then, the matrix
$\alpha I_q \oplus \beta I_q$ is the product of three involutions.
\end{lemma}

\begin{proof}
Set $\pi:=-\alpha\beta$.
If $\pi^q=1$, the result is already known by Lemma \ref{superdiagonalLemma}.

Assume now that $\pi^q=-1$. As $\F$ does not have characteristic $2$, this yields that $\pi$ has even order in the group $\F^*$,
and it ensues that $\pi^{2k+1} \neq 1$ for every integer $k$. Hence, by Lemma \ref{C1cycle},
$\alpha I_q \oplus \beta I_q$ is i-adjacent to $\calC_{q,1}(\pi)$. Besides, $\pi^{q}=\pm 1$, and hence the last statement in Lemma \ref{I2cycles}
shows that $\calC_{q,1}(\pi)$ is the product of two involutions. We conclude that $\alpha I_q \oplus \beta I_q$ is the product of three involutions.
\end{proof}

Now, we can prove Theorem \ref{theo3invol}.
Let $A \in \GL_n(\F)$ have determinant $\pm 1$.
Then, $M:=A \oplus I_n$ satisfies the conditions of Proposition \ref{VWPlemma}, and hence we have
non-negative integers $p,q,r$, a matrix $N \in \GL_p(\F)$ and a scalar $\alpha \in \F \setminus \{1\}$
such that
$$M\simeq N \oplus \alpha\,I_{q}\oplus I_{r}, \quad r \geq q,$$
and either $N$ is very-well-partitioned, or $N-I_p$ is nilpotent and $q=0$, or $N$ is void.
Noting that $I_{r-q}$ is the product of three involutions (say, three copies of itself),
we see that it suffices to consider the case when $r=q$.

If $N$ is void and $q>0$, then $\alpha^q=\det M=\pm 1$ and we directly deduce from Lemma \ref{diagonal3invollemmafinal}
that $M$ is the product of three involutions. If $N$ is void and $q=0$, then the result is obviously true.

If $N-I_p$ is nilpotent and $q=0$, then $M$ is triangularizable with sole eigenvalue $1$, and we deduce from
Theorem \ref{theo2} that it is the product of two involutions, and hence it is also the product of three involutions.

\vskip 3mm
In the rest of the proof, we assume that $N$ is very-well-partitioned.
If there is no integer $k \in \lcro 1,q\rcro$ such that $\alpha^k=\pm 1$,
then Proposition \ref{not3U2decompProp} readily yields that $M$ is the product of three involutions.

Assume now that there is a integer $k \in \lcro 1,q\rcro$ such that $\alpha^k=\pm 1$, and take the greatest such integer $a$.
Note that $\alpha^k \neq \pm 1$ for all $k \in \lcro 1,q-a\rcro$.
Then, split
$$M \simeq \underbrace{(N \oplus \alpha I_{q-a} \oplus I_{q-a})}_{M_1} \oplus \underbrace{(\alpha I_a \oplus I_a)}_{M_2}.$$
Note that $\det M_2=\pm 1$, and hence $\det M_1=\pm 1$.
By Lemma \ref{diagonal3invollemmafinal}, the matrix $M_2$ is the product of three involutions.
By Proposition \ref{not3U2decompProp} if $q-a>0$, and by Proposition \ref{wellpartdecomp} otherwise,
$M_1$ is the product of three involutions. We conclude that $M$ is the product of three involutions.

Theorem \ref{theo3invol} is now established.

\subsection{Unnatural extensions : additional results on simple matrices}

Here, we assume that there exists an element $i$ of $\F$ such that $i^2=-1$, and we fix such an element.

\begin{lemma}\label{iI2lemma1}
Let $k$ be a positive integer. Then, $i I_1 \oplus C_{2k-1}(i)$ is the product of three involutions, and also the product
of one involution and two $U_2$-matrices.
\end{lemma}

\begin{proof}
Note that $\det  C_{2k-1}(i)=\pm i$.
By Proposition \ref{cyclicfit1},
$$C_{2k-1}(i) \iadj C\bigl((t-1)^{2k-2}(t+i)\bigr) \simeq C_{2k-2}(1) \oplus C_1(- i),$$
and hence
$$i I_1 \oplus C_{2k-1}(i) \iadj C_{2k-2}(1) \oplus C_1(i) \oplus C_1(- i).$$
The latter matrix is obviously similar to its inverse and $-1$ is no eigenvalue of it, and hence
it is both the product of two involutions and the product of  two $U_2$-matrices. The conclusion ensues.
\end{proof}

\begin{cor}\label{iI2cor}
The matrix $i I_2$ is the product of three involutions, and also the product of
one involution and two $U_2$-matrices.
\end{cor}

\begin{lemma}\label{iI2lemma2}
Let $k$ be a positive integer. Then, $C_{2k}(i)$ is the product of three involutions, and also the product
of one involution and two $U_2$-matrices.
\end{lemma}

\begin{proof}
Set
$$K:=\begin{bmatrix}
1 & i \\
0 & -1
\end{bmatrix} \quad \text{and} \quad
L:=\begin{bmatrix}
-i & 1 \\
0 & i
\end{bmatrix},$$
and define
$$A:=K \oplus \cdots \oplus K \quad \text{and} \quad B:=i I_1 \oplus L \oplus \cdots \oplus L \oplus (-iI_1),$$
with $k$ copies of $K$ in the definition of $A$, and $k-1$ copies of $L$ in the one of $B$.
Then, one sees that $AB$ is upper-triangular with all its diagonal entries equal to $i$,
and for every pair $(u,v)\in \lcro 1,2k\rcro^2$ such that $v=u+1$, the entry of $AB$ at the $(u,v)$-spot is nonzero.
Hence, $AB-i I_{2k}$ is nilpotent with rank $2k-1$, and we deduce that $AB\simeq C_{2k}(i)$.
Obviously, $A$ is an involution, and $B$ is similar to its inverse and $-1$ is no eigenvalue of $B$.
Hence, $B$ is the product of two $U_2$-matrices.
The conclusion ensues that $C_{2k}(i)$ is the product of one involution and two
$U_2$-matrices, and hence it is also the product of three involutions.
\end{proof}

\begin{lemma}\label{wellpartdecompskew}
Let $N \in \GL_n(\F)$ be a very-well-partitioned matrix such that $\det N=\pm i$.
Then, $N \oplus i I_1$ is the product of three involutions, and it is also the product of one involution and two $U_2$-matrices.
\end{lemma}

\begin{proof}
Indeed, the Adaptation Theorem shows that $N$ is i-adjacent to $C\bigl((t-1)^{n-1}(t+i)\bigr)$, and hence
$N \oplus i I_1$ is i-adjacent to $B:=C_{n-1}(1) \oplus C_1(-i) \oplus C_1(i)$, a matrix which is similar to its inverse
and of which $-1$ is no eigenvalue. Hence, $B$ is both the product of two $U_2$-matrices and the product of two involutions, and
the conclusion ensues.
\end{proof}

\begin{lemma}\label{skew3invollastlemma1}
Let $\alpha,\beta$ be distinct nonzero scalars. Assume that, in the group $\F^*$, the element $(-\alpha\beta)$ has order $4q$ for some $q>0$.
Then, $\alpha I_q \oplus \beta I_q \oplus (iI_1)$ is the product of three involutions,
and it is also the product of one involution and two $U_2$-matrices.
\end{lemma}

\begin{proof}
Set $\pi=-\alpha\beta$, which has order $4q$. In particular $\pi^{2q}=-1$ (otherwise the order of $\pi$ would divide $2q$),
and hence $\pi^q=-\varepsilon i$ for some $\varepsilon \in \{-1,1\}$. Moreover,
there is no odd integer $l$ such that $\pi^l=-1$ (otherwise the order of $\pi$ would divide $2l$), and hence
$i \pi^{k+1} \neq -i \pi^{-k}$ for every integer $k$.
By Lemma \ref{C1cycle}, it follows that
$$\alpha I_q \oplus \beta I_q \iadj \underset{k=0}{\overset{q-1}{\bigoplus}}\,\bigl(C_1(-\varepsilon i \pi^{-k}) \oplus C_1(\varepsilon i \pi^{k+1})\bigr),$$
and hence
$$\alpha I_q \oplus \beta I_q \oplus iI_1 \iadj
B:=C_1(\varepsilon i) \oplus \underset{k=0}{\overset{q-1}{\bigoplus}}\,\bigl(C_1(-\varepsilon i \pi^{-k}) \oplus C_1(\varepsilon i \pi^{k+1})\bigr).$$
Noting that $\varepsilon i \pi^q=1$, we extract two blocks and we obtain
$$B \simeq C_1(1) \oplus \bigl[C_1(\varepsilon i) \oplus C_1(-\varepsilon i)\bigr] \oplus
\underset{k=1}{\overset{q-1}{\bigoplus}}\,\bigl(C_1(-\varepsilon i \pi^{-k}) \oplus C_1(\varepsilon i \pi^{k})\bigr).$$
This shows that $B$ is similar to its inverse.
Moreover, $-1$ is no eigenvalue of $B$: indeed, there can be no integer $k$ such that $|k|<q$ and $\pi^k=\pm i$,
otherwise $\pi^{4k}=1$ would yield $k=0$ which is absurd.
Hence, $B$ is the product of two $U_2$-matrices. The conclusion ensues.
\end{proof}

For the case of products of three involutions, we can generalize the previous result as follows:

\begin{lemma}\label{skew3invollastlemma2}
Let $\alpha,\beta$ be distinct nonzero scalars, and $p$ be a non-negative integer such that
$(-\alpha \beta)^p=\pm i$.
Then, $\alpha I_p \oplus \beta I_p \oplus i I_1$ is the product of three involutions.
\end{lemma}

\begin{proof}
Set $\pi:=-\alpha \beta$. Since the subgroup generated by $\pi$ contains an element of order $4$,
the order of $\pi$ is a multiple of $4$, which we write $4q$ for some $q>0$.
Since $\pi^{4p}=1$ and $\pi^{2p}=-1$, we find that $p$ is a multiple of $q$ but not of $2q$.
Hence, $p=2qm+q$ for some integer $m \geq 0$, and we deduce that
$\alpha I_p \oplus \beta I_p \oplus i I_1$ is similar to the direct sum of
$\alpha I_q \oplus \beta I_q \oplus i I_1$ and of $m$ copies of
$\alpha I_{2q} \oplus \beta I_{2q}$. Since $(-\alpha \beta)^q=\pm i$ and $(-\alpha \beta)^{2q}=-1$,
all those summands are products of three involutions (by Lemma \ref{skew3invollastlemma1} for the first summand, and
by Lemma \ref{diagonal3invollemmafinal} for the remaining ones), and the conclusion ensues.
\end{proof}

\subsection{Unnatural extensions: completing the proof}

We are ready to conclude the proof of Theorem \ref{skew3involutions}.
Assume that $\F$ contains an element $i$ such that $i^2=-1$.
Let $A \in \GL_n(\F)$ and $k \geq n$. Set $M:=A \oplus i I_k$ and assume that $\det M=\pm 1$.
We shall prove that $M$ is the product of three involutions.

By Proposition \ref{VWPcor}, there are non-negative integers $p,q,r$, a matrix $N \in \GL_p(\F)$,
and a scalar $\alpha \in \F \setminus \{i\}$ such that
$$M \simeq N \oplus \alpha I_q \oplus i I_r \quad \text{and} \quad r \geq q,$$
and either $N-iI_p$ is nilpotent and $q=0$, or $N$ is void, or
$N$ is very-well-partitioned.
Moreover, when $q=0$ and $N-i I_p$ is nilpotent, we can assume that $N$ has no Jordan cell of size $1$
(otherwise we put all those cells in the last $i I_r$ block).

Assume first that $q=0$ and that $N-iI_p$ is nilpotent with no Jordan cell of size $1$.
By the construction of $M$, we see that $r$ is greater than or equal to the number $s$ of Jordan cells of odd size of $N$.
Then, $1=\det M=\pm i^{r+s}$, and hence $r-s$ is even. By Corollary \ref{iI2cor}, the matrix $i I_{r-s}$
is the product of three involutions. We note that $M$ is similar to the direct sum of $i I_{r-s}$, of $s$ matrices of the form $C_{2k+1}(i)\oplus iI_1$
for some positive integer $k$, and of Jordan cells of even size for the eigenvalue $i$.
By Lemmas \ref{iI2lemma1} and \ref{iI2lemma2}, each one of those matrices is the product of three involutions, and hence so is $M$.

In the remainder of the proof, we assume that $N$ is either void or very-well-partitioned.
Since $i I_2$ is the product of three involutions, we further reduce the situation to the one where $r \in \{q,q+1\}$.
Assume that $N$ is void. Then, either $r=q$ and $(-i\alpha)^q =\pm 1$, in which case we use Lemma
\ref{diagonal3invollemmafinal} to see that $M$ is the product of three involutions, or
$r=q+1$ and $(-i\alpha)^q =\pm i$, in which case the same conclusion is reached by applying
Lemma \ref{skew3invollastlemma2}.

\vskip 3mm
It remains to deal with the case when $N$ is very-well-partitioned and $r\in \{q,q+1\}$.

Assume first that $q=0$. If $r=0$ then we deduce from Proposition \ref{wellpartdecomp} that $N$ is the product of three involutions.
If $r=1$, we get from the Adaptation Theorem that $N \iadj C\bigl((t-1)^{p-1}\,(t+i)\bigr) \simeq C_{p-1}(1) \oplus C_1(-i)$,
and hence $A \iadj C_{p-1}(1) \oplus C_1(-i) \oplus C_1(i)$. The latter matrix is the product of two involutions.

\vskip 3mm
Assume finally that $N$ is very-well-partitioned, $r\in \{q,q+1\}$ and $q>0$.
We split the discussion into two cases, whether $r=q$ or $r=q+1$.

\noindent \textbf{Case 1: $r=q$.} \\
If $(-i\alpha)^k \neq \pm 1$ for all $k \in \lcro 1,q\rcro$, then we readily deduce from Proposition \ref{not3U2decompProp} that
$M$ is the product of three involutions.
Assume now that $(-i\alpha)^k=\pm 1$ for some $k \in \lcro 1,q\rcro$, and denote by $a$ the greatest such integer.
Note that $(-i\alpha)^l \neq \pm 1$ for all $l \in \lcro 1,q-a\rcro$.
Let us split
$$M \simeq \underbrace{\bigl[N \oplus \alpha I_{q-a} \oplus i I_{q-a}\bigr]}_{M_1} \oplus \underbrace{\bigl[\alpha I_a \oplus i I_a\bigr]}_{M_2}$$
and note that $\det M_2=\pm 1$, and hence $\det M_1=\pm 1$. Then,
by Proposition \ref{not3U2decompProp} if $q-a>0$, and by Proposition \ref{wellpartdecomp} otherwise, we find that $M_1$
is the product of three involutions. Lemma \ref{diagonal3invollemmafinal} shows that $M_2$ is the product of three involutions,
and we conclude that so is $M$.

\vskip 3mm
\noindent \textbf{Case 2: $r=q+1$.} \\
If $(-i\alpha)^{4k} \neq 1$ for all $k \in \lcro 1,q\rcro$, then we directly deduce from Proposition \ref{3I2decomppropskew} that
$M$ is the product of three involutions.
Assume now that $(-i\alpha)^{4k}=1$ for some $k \in \lcro 1,q\rcro$, and denote by $a$ the greatest such integer.
Note then that  $(-i\alpha)^{4l} \neq 1$ for all $l \in \lcro 1,q-a\rcro$.
\begin{itemize}
\item Assume that $(-i\alpha)^a=\pm 1$. Then, we split
$$M \simeq \underbrace{\bigl[N \oplus \alpha I_{q-a} \oplus i I_{q-a} \oplus i I_1\bigr]}_{M_1} \oplus
\underbrace{\bigl[\alpha I_a \oplus i I_a\bigr]}_{M_2}.$$
Note that $\det M_2=\pm 1$ and hence $\det M_1=\pm 1$.
Then, $M_1$ is the product of three involutions, by Proposition \ref{3I2decomppropskew} if $q-a>0$, and
by Lemma \ref{wellpartdecompskew} otherwise. Besides, $M_2$ is the product of three involutions by Lemma \ref{diagonal3invollemmafinal}.

\item Assume that $(-i\alpha)^a=\pm i$. Then, we split
$$M \simeq \underbrace{\bigl[N \oplus \alpha I_{q-a} \oplus i I_{q-a}\bigr]}_{M_3} \oplus
\underbrace{\bigl[\alpha I_a \oplus i I_a \oplus i I_1\bigr]}_{M_4}.$$
Again, $\det M_4=\pm 1$ and $\det M_3=\pm 1$.
Then, $M_3$ is the product of three involutions, by Proposition \ref{not3U2decompProp} if $q-a>0$, and
by Proposition \ref{wellpartdecomp} otherwise. Besides, $M_4$ is the product of three involutions by Lemma \ref{skew3invollastlemma2}.
\end{itemize}
In any case, we conclude that $M$ is the product of three involutions. This completes the proof of
Theorem \ref{skew3involutions}.

\section{Products of two involutions and one unipotent matrix of index $2$}\label{1U22invSection}

In this short section, we assume that the field $\F$ does not have characteristic $2$,
and we prove Theorem \ref{theo3mixed1}, which we restate below:

\begin{quote}
Let $A \in \GL_n(\F)$ be such that $\det A=\pm 1$.
Then, the matrix $A \oplus I_n$ is the product of one $U_2$-matrix and two involutions.
\end{quote}

The strategy is identical to the one of the proof of Theorem \ref{theo3U2} given in Section \ref{3U2Section}, and hence we see that it suffices to prove the following result:

\begin{lemma}
Let $\alpha \in \F \setminus \{0,1\}$ and $q$ be a positive integer such that
$\alpha^q=\pm 1$. Then, $\alpha I_q \oplus I_q$ is the product of one $U_2$-matrix and two involutions.
\end{lemma}

\begin{proof}
If $\alpha^q=1$, the result is already known as part of Lemma \ref{superdiagonalLemma}.
In the rest of the proof, we assume that $\alpha^q=-1$.

Since $-1$ is a power of $\alpha$, we find that $\alpha$ has even order in the group $\F^*$.
In particular $\alpha^{k+1} \neq \alpha^{-k}$ for every integer $k$.
Hence, Lemma \ref{C1cycle} shows that
$$\alpha I_q \oplus I_q \uadj \calC_{q,1}(\alpha).$$
Since $\alpha^q=-1$, Lemma \ref{I2cycles} shows that $\calC_{q,1}(\alpha)$ is the product of two involutions. The conclusion ensues.
\end{proof}

Hence, Theorem \ref{theo3mixed1} is proved.

\section{Products of one involution and two unipotent matrices of index $2$}\label{1inv2U2Section}

In this section, we assume that the field $\F$ does not have characteristic $2$.

\subsection{Natural extensions}

Before we can prove Theorem \ref{theo3mixed2}, we need two consecutive lemmas.

\begin{lemma}\label{mixed2lastlemma1}
Let $\alpha \in \F \setminus \{0,1\}$. Assume that $-\alpha$ has even order $2q$
in the group $\F^*$.
Then, $\alpha I_q \oplus I_q$  is the product of one involution and two $U_2$-matrices.
\end{lemma}

\begin{proof}
Set $\pi:=-\alpha$ and note that $\pi^q=-1$. Since $\pi$ has even order, we have $\pi^{2k+1}\neq 1$ for every integer $k$.
Hence, Lemma \ref{C1cycle} yields that $\alpha I_{q-1} \oplus I_{q-1}$ is i-adjacent to $\calC_{q-1,1}(\pi)$.
Noting that $\pi^{-(q-1)}=-\pi=\alpha$, we deduce that
$$\alpha I_q \oplus I_q \iadj B:=C_1(1) \oplus C_1(\pi^{-(q-1)}) \oplus \calC_{q-1,1}(\pi).$$
We know from Lemma \ref{I2cycles} that $C_1(\pi^{-(q-1)}) \oplus \calC_{q-1,1}(\pi)$
is similar to its inverse, and hence so is $B$. Moreover, we see that $-1$ is no eigenvalue of $B$:
indeed, otherwise $\pi^k=-1$ for some $k$ such that $|k|<q$, which would yield $\pi^{2k}=1$
and then $k=0$ because $|2k|<2q$, leading to a contradiction. Hence, $B$ is the product of two $U_2$-matrices.
The conclusion ensues.
\end{proof}

\begin{lemma}\label{mixed2lastlemma2}
Let $\alpha \in \F \setminus \{0,1\}$, and $q$ be a positive integer such that $(-\alpha)^q=\pm 1$.
Then, the matrix $\alpha I_q \oplus I_q$ is the product of one involution and two $U_2$-matrices.
\end{lemma}

\begin{proof}
Set $\pi:=-\alpha$. If $\pi^q=1$, then the result readily follows from Lemma \ref{superdiagonalLemma}.

Assume now that $\pi^q=-1$. Hence, $\pi$ has even order, which we denote by $2p$,
and as $\pi^{2q}=1$ we find that $p$ divides $q$.
Hence, $\alpha I_q \oplus I_q$ is similar to the direct sum of copies of
$\alpha I_p \oplus I_p$, a matrix which is the product of one involution and two $U_2$-matrices
by Lemma \ref{mixed2lastlemma1}.
Hence, $\alpha I_q \oplus I_q$ is the product of one involution and two $U_2$-matrices.
\end{proof}

From there, the proof of Theorem \ref{theo3mixed2} is essentially similar to the one of Theorem \ref{theo3invol}.
The only difference is that one uses Lemma \ref{mixed2lastlemma2} instead of Lemma \ref{diagonal3invollemmafinal}.

\subsection{Unnatural extensions}

Here, we let $i$ be an element of $\F$ such that $i^2=-1$.
In order to prove Theorem \ref{skew1involution2unipotents}, we can adapt the strategy of the proof of Theorem
\ref{skew3involutions}, and we see that it suffices to prove the following result.

\begin{lemma}\label{mixed2skewlastlemma1}
Let $\alpha \in \F \setminus \{0,i\}$, and let $q$ be a positive integer such that
$(-i\alpha)^q=-1$. Then, $\alpha I_q \oplus i I_q$ is the product of one involution and two $U_2$-matrices.
\end{lemma}

In order to prove this result, a basic lemma is required:

\begin{lemma}\label{lastadjacencylemma}
Let $\alpha \in \F \setminus \{0,i\}$. Then,
$\alpha I_2 \oplus i I_1$ is i-adjacent to $C_1(-i) \oplus C_2(\alpha)$.
\end{lemma}

\begin{proof}[Proof of Lemma \ref{lastadjacencylemma}]
Set
$$S:=
\begin{bmatrix}
1 & 0 & 0 \\
1 & -1 & 0 \\
0 & 0 & 1
\end{bmatrix} \quad \text{and} \quad
A:=\begin{bmatrix}
\alpha & 0 & 0 \\
0 & i & 0 \\
0 & 1 & \alpha
\end{bmatrix}.$$
We see that $S$ is an involution and that $A \simeq \alpha I_2 \oplus i I_1$. Moreover,
$$SA =\begin{bmatrix}
\alpha & 0 & 0 \\
\alpha & -i & 0 \\
0 & 1 & \alpha
\end{bmatrix}$$
is obviously cyclic with characteristic polynomial $(t-\alpha)^2(t+i)$, to the effect that
$$SA \simeq C_1(-i) \oplus C_2(\alpha).$$
This yields the claimed result.
\end{proof}

\begin{proof}[Proof of Lemma \ref{mixed2skewlastlemma1}]
Set $\pi:=-i \alpha$, so that $\pi^q=-1$. Set $M:=\alpha I_q \oplus i I_q$.
We split the discussion into two cases, whether $q$ is even or odd.

\vskip 3mm
\noindent \textbf{Case 1: $q$ is even.} \\
We write $q=2p$. As $(\pi^{-p})^2=-1$, we find
$\pi^{-p}=\varepsilon i$ for some $\varepsilon \in \{1,-1\}$.
By Lemma \ref{C2basic}, we have
$$\alpha I_{2p-2} \oplus i I_{2p-2} \iadj \underset{k=0}{\overset{p-2}{\bigoplus}}\bigl(C_2(\varepsilon \pi^{-k}) \oplus C_2(\varepsilon \pi^{k+1})\bigr).$$
Combining this with Lemma \ref{lastadjacencylemma}, we deduce that
$$M \iadj B:= \bigl[C_1(i)\oplus C_1(-i)\bigr] \oplus C_2(i\pi) \oplus
\underset{k=0}{\overset{p-2}{\bigoplus}}\bigl(C_2(\varepsilon \pi^{-k}) \oplus C_2(\varepsilon \pi^{k+1})\bigr).$$
Noting that $(\varepsilon \pi^{p-1})(i\pi)=1$, we see that $B$ is similar to its inverse. Moreover,
$B$ has no Jordan cell of odd size for the eigenvalue $-1$, and hence it is the product of two $U_2$-matrices.
Therefore, $M$ is the product of one involution and two $U_2$-matrices.

\vskip 3mm
\noindent \textbf{Case 2: $q$ is odd.} \\
If $q=1$, then $\alpha=-i$ and $M$ is the product of two $U_2$-matrices.
In the remainder of the proof, we assume that $q>1$.

Let us write $q=2p+1$ for some positive integer $p$.
Note that $(i \pi^{p+1})^2=\pi$. Hence,
Lemma \ref{basicblocklemma} shows that $i I_1 \oplus \alpha I_1 \iadj C_2(i \pi^{p+1})$.
Moreover, Lemma \ref{C2basic} shows that
$$i I_{2p-2} \oplus \alpha I_{2p-2} \iadj \underset{k=1}{\overset{p-1}{\bigoplus}}\bigl(C_2(-i \pi^{-k}) \oplus C_2(i \pi^{k+1})\bigr).$$
It follows from Lemma \ref{lastadjacencylemma} that
$$M \iadj B:=C_1(i)\oplus C_1(-i) \oplus C_2(i\pi) \oplus C_2(i\pi^{p+1})\oplus
\underset{k=1}{\overset{p-1}{\bigoplus}}\bigl(C_2(-i \pi^{-k}) \oplus C_2(i \pi^{k+1})\bigr).$$
Reorganizing the terms and noting that $i \pi^{p+1}=-i \pi^{-p}$, we obtain
$$B \simeq \bigl[C_1(i)\oplus C_1(-i)\bigr] \oplus \underset{k=1}{\overset{p}{\bigoplus}}\bigl(C_2(-i \pi^{-k}) \oplus C_2(i \pi^{k})\bigr).$$
Hence, $B$ is similar to its inverse and has no Jordan cell of odd size for the eigenvalue $-1$ (in fact, one can prove that $-1$ is no eigenvalue of $B$).
Thus, $B$ is the product of two $U_2$-matrices, and hence $M$ is the product of one involution and two $U_2$-matrices.
\end{proof}

This completes the proof of Theorem \ref{skew1involution2unipotents}.

\section{Optimality issues}\label{optimalitySection}

Here, we briefly discuss the optimality of our results.
For example, if we refer to Theorem \ref{theo3invol}, the problem is the following one:
Given a positive integer $n$, what is the minimal integer $k \geq 0$ such that,
for any field $\F$ and any matrix $A \in \GL_n(\F)$ with determinant $\pm 1$, the matrix $A \oplus I_k$
is the product of three involutions?
It turns out that the solution $n$ is not optimal but very close to optimality.
This is due to the fact that there is room for improvement in the lemmas that deal with matrices of the form
$\alpha I_q \oplus \beta I_q$: there, we tried to keep things as general as possible and in particular we seldom cared
about the value of $\beta$. Yet, if we assign a specific value to $\beta$, say $\beta=1$ for Theorem \ref{theo3invol},
then there is room for improvement, as we will now see.

Here, we state the optimal results without proof:

\begin{itemize}
\item For every integer $n \geq 2$ and every $A \in \SL_n(\F)$, the matrix $A \oplus I_{n-2}$
is the product of three $U_2$-matrices. However, if $n \geq 3$ then for every $\alpha \in \F \setminus \{0\}$
of order $n$, the matrix $\alpha I_n \oplus I_{n-3}$ is not the product of three $U_2$-matrices.

\item For every integer $n \geq 2$ and every $A \in \GL_n(\F)$ having determinant $\pm 1$, the matrix $A \oplus I_{n-1}$
is the product of one $U_2$-matrix and two involutions. However, for every $\alpha \in \F \setminus \{0\}$ of order $2n$,
the matrix $\alpha I_n \oplus I_{n-2}$ is not the product of one $U_2$-matrix and two involutions.

\item For every even integer $n \geq 4$ and every $A \in \GL_n(\F)$ having determinant $\pm 1$, the matrix $A \oplus I_{n-2}$
is the product of one involution and two $U_2$-matrices; however for every $\alpha \in \F \setminus \{0\}$ of order $2n$,
the matrix $\alpha I_n \oplus I_{n-3}$ is not the product of three involutions.

\item For every odd integer $n \geq 3$ and every $A \in \GL_n(\F)$ having determinant $\pm 1$, the matrix $A \oplus I_{n-1}$
is the product of one involution and two $U_2$-matrices; however for every $\alpha \in \F \setminus \{0\}$ of order $2n$,
the matrix $\alpha I_n \oplus I_{n-2}$ is not the product of three involutions.
\end{itemize}

In those results, the positive statement can be proved by the same techniques we have resorted to in the present article
(using cycles of small companion matrices), whereas the negative statement requires a deep understanding of the structure of products
of two quadratic matrices (see \cite{dSPprod2}).

\vskip 3mm
Now, let us turn to unnatural extensions.
Assume that $\F$ does not have characteristic $2$, let $n>2$ be an integer, and let $A \in \GL_n(\F)$.
We start with decompositions into the product of three $U_2$-matrices.

\begin{itemize}
\item If $\det A=1$ and $n$ is even, then $A \oplus (-I_{n-2})$ is the product of three $U_2$-matrices.
If $n$ is even and not a multiple of $4$, then for any $\alpha \in \F \setminus \{0\}$ of order $\frac{n}{2}$,
the matrix $\alpha I_n \oplus (-I_{n-4})$ is not the product of three $U_2$-matrices.
If $n$ is a multiple of $4$, then for any $\alpha \in \F \setminus \{0\}$ of order $n$,
the matrix $\alpha I_n \oplus (-I_{n-4})$ is not the product of three $U_2$-matrices.
\item If $\det A=-1$ and $n$ is odd, then $A \oplus (-I_{n-2})$ is the product of three $U_2$-matrices.
If $n$ is odd then, for any $\alpha \in \F \setminus \{0\}$ of order $2n$, the matrix
$\alpha I_n \oplus (-I_{n-4})$ is not the product of three $U_2$-matrices.
\item If $\det A=-1$ and $n$ is even, then $A \oplus (-I_{n-1})$ is the product of three $U_2$-matrices.
If $n$ is even, then for any $\alpha \in \F \setminus \{0\}$ of order $2n$, the matrix $\alpha I_n \oplus (-I_{n-3})$ is not the product of three $U_2$-matrices.
\item If $\det A=1$ and $n$ is odd, then $A \oplus (-I_{n-1})$ is the product of three $U_2$-matrices.
If $n$ is odd then, for any $\alpha \in \F \setminus \{0\}$ of order $n$, the matrix $\alpha I_n \oplus (-I_{n-3})$ is
not the product of three $U_2$-matrices.
\end{itemize}

\vskip 2mm
We finish with decompositions into the product of one involution and two $U_2$-matrices (or three involutions).
To this end, we let $i$ be an element of order $4$ in $\F \setminus \{0\}$.
\begin{itemize}
\item If $\det A=\pm i$ and $n$ is odd, then $A \oplus i I_n$ is the product of one involution and two $U_2$-matrices.
However, if $n$ is odd, then for any $\alpha \in \F \setminus \{0\}$ of order $4n$ such that $-i\alpha$ is of order $2n$, the matrix $\alpha I_n \oplus i I_{n-2}$
is not the product of three involutions. Note that such a scalar $\alpha$ exists in the field of complex numbers:
it suffices to choose a complex number $\pi$ of order $2n$, and to take $\alpha:=i\pi$.

\item If $\det A=\pm i$ and $n$ is even, then $A \oplus i I_{n-1}$ is the product of one involution and two $U_2$-matrices.
However, if $n$ is even, then for any $\alpha \in \F \setminus \{0\}$ of order $4n$, the matrix $\alpha I_n \oplus i I_{n-3}$
is not the product of three involutions.

\item If $\det A=\pm 1$ and $n$ is odd, then $A \oplus i I_{n-1}$ is the product of one involution and two $U_2$-matrices.
However, if $n$ is odd, then for any $\alpha \in \F \setminus \{0\}$ of order $2n$, the matrix $\alpha I_n \oplus i I_{n-3}$
is not the product of three involutions.

\item If $\det A=\pm 1$ and $n$ is even, then $A \oplus i I_{n-2}$
is the product of one involution and two $U_2$-matrices.
However, if $n$ is even and greater than $2$, then for any $\alpha \in \F \setminus \{0\}$ of order $2n$ such that
$-i\alpha$ is not of order $\frac{n}{2}$, the matrix $\alpha I_n \oplus i I_{n-4}$
is not the product of three involutions.
Note that such a scalar $\alpha$ exists in the field of complex numbers:  either $\frac{n}{2}$ is odd, and
then it suffices to start from an element $\pi$ of order $n$ and to take $\alpha:=i\pi$, or $\frac{n}{2}$ is even and
it suffices to choose $\alpha$ of order $2n$.
\end{itemize}

\end{document}